\documentclass[10pt,a4paper]{article}
\usepackage{amsthm}
\usepackage[pdftex]{graphicx} 
\usepackage{amsmath}
\usepackage{amssymb}
\usepackage{mathtools}
\usepackage{ascmac}
\usepackage{mathrsfs}
\usepackage{setspace}
\usepackage{nccmath}
\usepackage{color}
\usepackage{enumerate}
\usepackage[rm]{titlesec}
\usepackage{aliascnt} 
\usepackage{indentfirst}
\usepackage[english]{babel}   

\titleformat*{\section}{\large\bfseries}
\titleformat*{\subsection}{\bfseries}


\theoremstyle{plain} 
\newtheorem{theorem}{Theorem}[section]
\newtheorem*{theorem*}{Theorem}
\newtheorem{lemma}{Lemma}[section]
\newtheorem*{lemma*}{Lemma}
\newtheorem{proposition}{Proposition}[section]
\newtheorem*{proposition*}{Propposition}
\newtheorem*{acknowledgements}{Acknowledgements}

\theoremstyle{definition}

\newtheorem*{definition*}{Definition}

\newtheorem*{property*}{性質}
\newtheorem{remark}{Remark}[section]
\newtheorem*{remark*}{Remark}
\newtheorem{example}{Example}
\newtheorem*{example*}{Example}
\newtheorem{assume}{Assumption}[section]
\newtheorem*{assume*}{Assumption}
\title{Information geometry of warped product spaces}
\author{Yasuaki Fujitani\thanks{Department of Mathematics, Osaka University, Osaka 560-0043, Japan (\texttt{u197830k@ecs.osaka-u.ac.jp})}}
\begin{document}
\maketitle
\begin{abstract}
    Information geometry is an important tool to study statistical models. There are some important examples in statistical models which are regarded as warped products. In this paper, we study 
    information geometry of warped products. We consider the case where the warped product and its fiber space are equipped with dually flat connections and, in the particular case of a cone, characterize the connections on the base space $\mathbb{R}_{>0}$. The resulting connections turn out to be the $\alpha$-connections with $\alpha = \pm{1}$.
\end{abstract}
\tableofcontents
\section{Introduction}
Recently, the study of spaces consisting of probability measures is getting more attention. As tools to investigate such spaces, there are two famous theories in geometry:
information geometry and Wasserstein geometry. Information geometry is mainly concerned with finite dimensional statistical models and Wasserstein geometry is concerned with infinite dimensional spaces of probability measures.
We can compare these two geometries, for example, on Gaussian distributions.

This paper concerns information geometry of a warped product and, in particular, on a cone, which is a kind of warped product of the line $\mathbb{R}_{>0}$ and a manifold. 
Under some natural assumptions, we characterize connections on the line, with which warped products are constructed.
The assumption we set is different from that of \cite{leo} and matches examples of statistical models.

Examples of warped product metrics include the denormalizations of the Fisher metric, the Bogoliubov-Kubo-Mori metric and the Fisher metric on the Takano Gaussian space, which is a set of multivariate Gaussian distributions with restricted parameters. 
Besides these examples, there are some more statistical models represented as warped products.
In \cite{takatsu2}, it was shown that the Wasserstein Gaussian space, which is the set of multivariate Gaussian distributions on $\mathbb{R}^n$ with mean zero equipped with the $L^2$-Wasserstein metric, has a cone structure and in \cite{location} the relations between Fisher metrics of location scale models and warped product metrics are studied.
It seems that warped products get more attention in the field of statistical models than before.

Although information geometry is studied on real manifolds, the theory of statistical manifolds is studied in the field of affine geometry and statistical structures on complex manifolds get more attention as in \cite{furuhata1}.
Also in this field, warped products are important since they play an important role in the theory of submanifolds in complex manifolds, for example, CR submanifold theory as in \cite{chen1}.
There are many researches extending the theories of CR submanifolds in K\"{a}hler manifolds to submanifolds in holomorphic statistical manifolds as in \cite{chen2}.
Statistical structures in \cite{furuhata2} and the structures cultivated in this paper are slightly different because we do not care the compatibility of statistical structures and complex structures. This compatibility is expressed in the definition of holomorphic statistical structures in \cite{furuhata1}.

This paper is organized as follows. In Section 2, we briefly review information geometry. Section 3 is devoted to some formulas in warped products. Then in Section 4, we study cones and consider necessary conditions for making both the cone and the fiber space to
be dually flat. This necessary condition states that there are only two possible connections on the line. The following theorem is one of our main results.
\begin{theorem*}\rm{(Theorem 4.1)}
    Under Assumption \ref{assume2}, we have 
    \begin{equation*}
        D_{\partial_t}\partial_t = \frac{1}{t}\frac{\partial}{\partial t} \mbox{ or } -\frac{1}{t}\frac{\partial}{\partial t},
    \end{equation*}
    where $t$ is the natural coordinate on the line $\mathbb{R}_{>0}$, which is the base space of the warped product.
\end{theorem*}

By observing examples, these two connections turn out to be the $\alpha$-connections with $\alpha = \pm{1}$.
An analogous characterization for the Takano Gaussian space is also considered in Section 5. In Section 6, we discuss dually flat connections on the Wasserstein Gaussian space. We remark that, although it is known in \cite{tayebi} that there is no dually flat proper doubly warped Finsler manifold, what they actually proved is that some coordinates cannot be dual affine coordinates. 
Thus our claims do not contradict their claim.  We also discuss this point in Section 6.
In Section 7, we study two-dimensional warped products as an appendix.

\section{Preliminaries}
\subsection{Information geometry}\label{infogeo}
We briefly review the basics of information geometry, we refer to \cite{amari} for further reading.
Let $(M,g)$ be a Riemannian manifold and $\nabla$ be an affine connection of $M$. 
$\mathfrak{X}(M)$ denotes the set of $C^\infty$ vector fields on $M$.
We define another affine connection $\nabla^*$ by
\begin{equation*}
    Xg(Y,Z) = g(\nabla_XY,Z) + g(Y,\nabla^*_XZ)
\end{equation*}
for $X,Y,Z\in\mathfrak{X}(M)$. We call $\nabla^*$ the \emph{dual connection} of $\nabla$. We define the torsion and the curvature of $\nabla$ by
\begin{equation*}
    T(X,Y) := \nabla_XY - \nabla_YX - [X,Y],\quad
    R(X,Y)Z := [\nabla_X,\nabla_Y]Z - \nabla_{[X,Y]}Z,
\end{equation*}
respectively. 
If $R$ satisfies
\begin{equation*}
    R(X,Y)Z = k\{g(Y,Z)X - g(X,Z)Y\}
\end{equation*}
for some $k\in\mathbb{R}$ and all $X,Y,Z\in\mathfrak{X}(M)$, $(M,g,\nabla)$ is called a space of constant curvature $k$. We summarize some important facts on $\nabla$ and $\nabla^*$ in the following.
\begin{proposition}\label{levi}
    Let $\nabla$ and $\nabla^*$ be dual affine connections of $M$. If two of the following conditions hold true, then the other two of them also hold true:
    \begin{itemize}
        \item $\nabla$ is torsion free,
        \item $\nabla^*$ is torsion free,
        \item $\nabla g$ is a symmetric tensor,
        \item $\frac{\nabla + \nabla^*}{2}$ is the Levi-Civita connection of $g$.
    \end{itemize}
\end{proposition}
\begin{proposition}\label{curvature_prop}
    Let $(M,g,\nabla,\nabla^*)$ be a Riemannian manifold with dual affine connections. 
    The curvature with respect to $\nabla$ vanishes if and only if the curvature with respect to $\nabla^*$ vanishes.
\end{proposition}

    Let $(M,g,\nabla,\nabla^*)$ be a Riemannian manifold with dual affine connections. If the torsion and the curvature with respect to $\nabla$
    and those of $\nabla^*$ all vanish, then we say that $(M,g,\nabla,\nabla^*)$ is \emph{dually flat}.
 For a local coordinate system $(U; x_1,\cdots, x_n)$, if the Christoffel symbols $\{\Gamma^k_{ij}\}$ of $\nabla$
    vanish, we call it \emph{$\nabla$-affine coordinates}.

\begin{proposition}
    Let $(M,g,\nabla,\nabla^*)$ be a Riemannian manifold with dual affine connections. If it is dually flat, then there exist $\nabla$-affine coordinates $(x_i)$ and $\nabla^*$-affine coordinates $(y_j)$ such that 
    \begin{equation*}
        g\left(\frac{\partial}{\partial x_i},\frac{\partial}{\partial y_j}\right) = \delta_{ij}.
    \end{equation*}
\end{proposition}

The coordinates $\{(x_i),(y_j)\}$ above are called \emph{dual affine coordinates}. 
Using dual affine coordinates, we can construct the canonical divergence [2, \S 3.4].

Next, we introduce the Fisher metric and $\alpha$-connections. 
Consider a family $\mathcal{S}$ of probability distributions on a finite set $\mathcal{X}$.
Suppose that $\mathcal{S}$ is parameterized by $n$ real-valued variables $[\xi^1,\ldots,\xi^n]$ so that
\begin{equation*}
    \mathcal{S} := \{p_\xi = p(x;\xi)\mid \xi = [\xi^1,\ldots,\xi^n]\in\Xi\},
\end{equation*}
where $\Xi$ is an open subset of $\mathbb{R}^n$.

For $\alpha \in\mathbb{R}, u>0, x\in\mathcal{X}$ and $\xi\in\Xi$, we put
\begin{equation*}
    L^{(\alpha)}(u) := \begin{cases}
        \frac{2}{1-\alpha}u^{\frac{1-\alpha}{2}} & (\alpha\neq 1),\\
        \log u  & (\alpha = 1),
    \end{cases}
\quad
    l^{(\alpha)}(x;\xi) := L^{(\alpha)}(p(x;\xi)).
\end{equation*}
Then, we define the \emph{Fisher metric} $g$ as 
    \begin{equation*}
        g_{ij}(\xi) := \int\partial_il^{(\alpha)}(x;\xi)\partial_jl^{(-\alpha)}(x;\xi)\, dx,
    \end{equation*}
 and \emph{$\alpha$-connections} $\nabla^{(\alpha)}$ as 
    \begin{equation*}
        \Gamma^{(\alpha)}_{ij,k}(\xi) := \int\partial_i\partial_j l^{(\alpha)}(x;\xi)\partial_kl^{(-\alpha)}(x;\xi)\, dx,
    \end{equation*}
where $g(\nabla^{(\alpha)}_{\partial_i}\partial_j,\partial_k) = \Gamma^{(\alpha)}_{ij,k}$.
Note that the Fisher metric does not depend on $\alpha$. We set 
\begin{equation*}
    \tilde{\mathcal{S}} := \{\tau p_\xi \mid \xi \in\Xi , \tau > 0\},
\end{equation*}
and call it the \emph{denormalization} of $\mathcal{S}$.

    In \cite{amari}, the Fisher metric and connections on $\tilde{\mathcal{S}}$ are defined as follows. An extension $\tilde{l}$ of $l$ is defined as
    \begin{equation*}
        \widetilde{l}^{(\alpha)} = \widetilde{l}^{(\alpha)}(x;\xi,\tau) := L^{(\alpha)}(\tau p(x;\xi)).
    \end{equation*}
    Using this $\tilde{l}$, we define the metric and connections on $\widetilde{\mathcal{S}}$ by
    \begin{equation}\label{netric_denormalization}
        \tilde{g}_{ij}(\xi) := \int\partial_i\tilde{l}^{(\alpha)}\partial_j\tilde{l}^{(-\alpha)}\, dx,\quad\tilde{\Gamma}^{(\alpha)}_{ij,k} = \int\partial_i\partial_j \tilde{l}^{(\alpha)}\partial_k\tilde{l}^{(-\alpha)}\, dx.
    \end{equation}
\subsection{Quantum information geometry}
Information geometry of density matrices is called quantum information geometry. 
The set of density matrices $\mathcal{D}$ is defined as
\begin{equation*}
    \mathcal{D} := \{\rho\in\mathbb{P}(n) | \mathrm{Tr}(\rho) = 1\},
\end{equation*}
where $\mathbb{P}(n)$ is the set of $n\times n$ positive definite Hermitian matrices. 
Parameterizing elements of $\mathcal{D}$ as $\rho_\xi$ by $\xi\in\Xi$, the $m$-representation of the natural basis is written as 
\begin{equation*}
    (\partial_i)^{(m)} = \partial_i\rho.
\end{equation*}
The \emph{mixture connection} $\nabla^{(m)}$ is a connection such that
\begin{equation*}
    (\nabla^{(m)}_{\partial_i}\partial_j)^{(m)} = \partial_i\partial_j\rho.
\end{equation*}
We set
\begin{equation*}
    \mathcal{MON} := \left\{f:\mathbb{R}_{>0}\rightarrow \mathbb{R}_{>0} | f \mbox{ is operator monotone}, \,f(1) = 1, f(t) = tf\left(\frac{1}{t}\right)\right\}.
\end{equation*} 
The \emph{monotone metric} for $f\in\mathcal{MON}$ is expressed as 
\begin{equation*}
    g^f_\rho(X,Y) = \mbox{Tr}\left\{X^*\frac{1}{(2\pi i)^2}\oint\oint c(\xi,\eta) \frac{1}{\xi-\rho} Y \frac{1}{\eta-\rho} \, d\xi d\eta \right\},
\end{equation*}
where $c(x,y) = 1/(yf(x/y))$ and $\xi(t), \eta(t)$ are paths surrounding the positive spectrum of $\rho$. The monotone metric for $f(x) = (x-1)/\log x$ is called the Bogoliubov-Kubo-Mori (BKM) metric.  We refer to \cite{amari} and \cite{dit} for further reading.
It is known that the BKM metric enjoys the following remarkable property.
\begin{proposition}
    $(\mathcal{D},\rm{BKM})$ equipped with the mixture connection is dually flat.
\end{proposition}

We can find a proof of this proposition in [2, Theorem 7.1], and the proof does not use the condition that the matrices considered have trace 1. Thus
we can prove the proposition below in completely the same way.

\begin{proposition}
    $(\mathbb{P}(n),\rm{BKM})$ equipped with the mixture connection is dually flat.
\end{proposition}

\subsection{Takano Gaussian space}\label{takano_gauss}
In this subsection, we explain some results from \cite{takano}. We consider multivariate Gaussian distributions
\begin{equation*}
    p(x;\xi) = \frac{1}{(\sqrt{2\pi}\sigma)^n}\prod_{i=1}^n \exp\left\{-\frac{(x_i-m_i)^2}{2\sigma^2}\right\},
\end{equation*}
where $\xi = (\sigma,m_1,\ldots , m_n)\in L^{(n+1)}, \,L^{(n+1)} := \mathbb{R}_{>0}\times \mathbb{R}^n$.

By a straightforward calculation, we obtain the Fisher metric $G$ as
\begin{equation*}
    G_{\sigma\sigma} = \frac{2n}{\sigma^2},\quad G_{\sigma i} = G_{i \sigma} = 0,\quad G_{ij} = \frac{1}{\sigma^2}\delta_{ij},
\end{equation*}
where $\partial_\sigma = \partial/\partial\sigma$ and $\partial_i = \partial/\partial m_i$, i.e., 
\begin{equation*}
    ds^2 = \frac{1}{\sigma^2}(2nd\sigma^2 + dm_1^2 + \cdots + dm_n^2).
\end{equation*}
Its $\alpha$-connections are
\begin{equation*}
    \Gamma^{(\alpha)}_{ij,k} = 0, \quad \Gamma_{ij,\sigma}^{(\alpha)} = \frac{1-\alpha}{\sigma^3}\delta_{ij},\quad \Gamma^{(\alpha)}_{i\sigma,k} = -\frac{1 + \alpha}{\sigma^3}\delta_{ik},
\end{equation*}
\begin{equation*}
    \Gamma^{(\alpha)}_{i\sigma,\sigma} = 0,\quad \Gamma^{(\alpha)}_{\sigma\sigma,i} = 0,\quad \Gamma^{(\alpha)}_{\sigma\sigma,\sigma} = -(1 + 2\alpha)\frac{2n}{\sigma^3},
\end{equation*}
and 
\begin{equation*}
    \nabla^{(\alpha)}_{\partial_i}\partial_j = \frac{1-\alpha}{2n\sigma}\delta_{ij}\partial_\sigma,\quad \nabla^{(\alpha)}_{\partial_i}\partial_\sigma = \nabla^{(\alpha)}_{\partial_\sigma}\partial_i = -\frac{1 + \alpha}{\sigma}\partial_i,\quad \nabla^{(\alpha)}_{\partial_\sigma}\partial_\sigma = -\frac{1 + 2\alpha}{\sigma}\partial_\sigma.
\end{equation*}

In \cite{takano}, they call \emph{$\alpha$-flat} if the curvature tensor with respect to the $\alpha$-connection vanishes identically and the following fact is proved.
\begin{proposition}
    $(L^{(n + 1)},ds^2,\nabla^{(\alpha)})$ is a space of constant curvature $-\frac{(1-\alpha)(1 + \alpha)}{2n}$. In particular,
    $(L^{(n + 1)},ds^2)$ is $(\pm{1})$-flat.
\end{proposition}

For simplicity, we call $(L^{(n + 1)},ds^2)$ the \emph{Takano Gaussian space} in this paper.

\section{Warped products}\label{cone}
In this section, we calculate dual affine connections on warped products.
\subsection{Koszul formula}
Let $\nabla$, $\nabla^*$ be torsion free dual affine connections on $(M,g)$.
For $X,Y,Z,W \in \mathfrak{X}(M)$, let us first see a kind of Koszul formula for $\nabla$. Summing up 
\begin{eqnarray*}
    Xg(Y,Z) &=& g(\nabla_XY,Z) + g(Y,\nabla^*_XZ)\nonumber,\\
    Yg(X,Z) &=& g(\nabla_YX,Z) + g(X,\nabla^*_YZ)\nonumber,\\
    -Zg(X,Y) &=& -g(\nabla_ZX,Y) - g(X,\nabla^*_ZY),
\end{eqnarray*}
we get
\begin{eqnarray*}
    Xg(Y,Z) + Yg(X,Z) - Zg(X,Y) &=& g(\nabla_XY,Z) + g(\nabla_YX,Z) + g(Y,\nabla^*_XZ - \nabla_ZX) + g(X,\nabla^*_YZ - \nabla^*_ZY).
\end{eqnarray*}
Recalling that we consider torsion free affine connections, we have
\begin{equation}\label{star}
    2g(\nabla_XY,Z) = Xg(Y,Z) + Yg(X,Z) - Zg(X,Y) + g([X,Y],Z) - g(Y,\nabla^*_XZ-\nabla_ZX) - g(X,[Y,Z]).
\end{equation}
In order to have a further look on $(\nabla^*_XZ-\nabla_ZX)$, we put
\begin{equation*}
    a(X,W) := \nabla^*_XW - \nabla_WX,
\end{equation*}
and calculate 
\begin{eqnarray*}
    a(X,W) - a(W,X) &=& (\nabla^*_XW-\nabla_WX) - (\nabla^*_WX - \nabla_XW) = 2[X,W],\nonumber\\
    a(X,W) + a(W,X) &=& (\nabla^*_XW - \nabla_XW) + (\nabla^*_WX - \nabla_WX) = -2\left(P_XW + P_WX\right),
\end{eqnarray*}
where we put
\begin{equation*}
    P := \frac{\nabla-\nabla^*}{2}.
\end{equation*}
Let us collect some properties of $P$.
\begin{lemma}
    Let $f$ be an arbitrary $C^\infty$ function on $M$. For any $X,Y,Z\in\mathfrak{X}(M)$, we have the following equations:
    \begin{equation}
        \label{aa}
        P_XY = P_YX,
    \end{equation}
    \begin{equation}
        \label{bb}
        g(P_XY,Z) = g(Y,P_XZ),
    \end{equation}
    \begin{equation}
        \label{cc}
        P_{fX} Y = fP_XY,\quad P_XfY = fP_XY.
    \end{equation}
    \begin{proof}
    For (\ref{aa}),
        \begin{equation*}
            P_XY - P_YX = \frac{\nabla_XY-\nabla_YX}{2} - \frac{\nabla^*_XY - \nabla^*_YX}{2} = \frac{1}{2}([X,Y]-[X,Y]) = 0.
        \end{equation*}
    For (\ref{bb}),
    \begin{eqnarray*}
        g\left(\frac{\nabla_X-\nabla^*_X}{2}Y,Z\right) &=& \frac{1}{2}\left(g(\nabla_XY,Z)-g(\nabla^*_XY,Z)\right)\nonumber\\
        &=& \frac{1}{2}\left(Xg(Y,Z)-g(Y,\nabla^*_XZ)\right)-\frac{1}{2}(Xg(Y,Z)-g(Y,\nabla_XZ))\nonumber\\
        &=& \frac{1}{2}\left(g(Y,\nabla_XZ)-g(Y,\nabla^*_XZ)\right)\nonumber\\
        &=& g(Y,P_XZ).
    \end{eqnarray*}
    For (\ref{cc}), the first equation is clear and we also observe
    \begin{equation*}
        P_X(fY) = \frac{\nabla_X(fY) - \nabla^*_X(fY)}{2} = Xf\frac{Y - Y}{2} + f\frac{\nabla_XY - \nabla^*_XY}{2} = fP_XY.
    \end{equation*}
\end{proof}
\end{lemma}
By the above lemma, we can express $a$ using $P$ as
\begin{equation*}
    a(X,W) + a(W,X) = -2(P_XW + P_WX) = -4P_XW,\quad a(X,W) = [X,W] - 2P_XW.
\end{equation*}
Substituting this into (\ref{star}), we obtain the following Koszul formula: 
\begin{equation}
    2g(\nabla_XY,Z) = Xg(Y,Z) + Yg(X,Z) - Zg(X,Y) + g([X,Y],Z) - g(Y,[X,Z] - 2P_XZ) - g(X,[Y,Z]).\label{starstar}
\end{equation}

\subsection{O'Neill formulas for affine connections on warped products} 
Let $(B,g_B),(F,g_F)$ be Riemannian manifolds, $f$ be a positive $C^{\infty}$-function on $B$,
$M := B\times_f F$ be the warped product of them equipped with the metric $G := g_B + f^2 g_F$, and $D,D^*$ be torsion free dual affine connections on $M$.
Denote by $\mathcal{L}(F)$, $\mathcal{L}(B)$ the sets of lifts of vector fields on $F$ to $M$, $B$ to $M$, respectively.
Let $X,Y,Z\in\mathcal{L}(B)$, $U,V,W\in\mathcal{L}(F)$ in the sequel.
We will assume the following.
\begin{assume}\label{assume1}
    $D_XY\in\mathcal{L}(B)$ for all $X,Y \in\mathcal{L}(B)$.
\end{assume}
\begin{lemma}\label{3_2}
    Under Assumption \ref{assume1}, we have
    \begin{equation*}
        G(D^*_XY,V) = 0,\quad G(P_XV,Y) = 0.
    \end{equation*}
\end{lemma}
    \begin{proof}
        The first equation follows from Assumption \ref{assume1} and the fact that $(D + D^*)/2$ is the Levi-Civita connection (recall Proposition \ref{levi}). 
        To see the second equation, since $G(Y,V) = G(X,V) = 0$ and $[X,V] = [Y,V] = 0$, we have
    \begin{eqnarray*}
        2G(D_XY,V) &=& XG(Y,V) + YG(X,V) - VG(X,Y) + G([X,Y],V) - G(Y,[X,V]-2P_XV) - G(X,[Y,V])\nonumber\\
        &=& -VG(X,Y) + G([X,Y],V) + G(Y,2P_XV)\nonumber\\
        &=& G(Y,2P_XV).
    \end{eqnarray*}
By combining this with $2G(D_XY,V) = 0$ by Assumption \ref{assume1}, the second equation holds. 
    \end{proof}
Let us modify some formulas on warped products in \cite{oneil} for the Levi-Civita connections to those for affine connections.

First we express $D_VX$.
On the one hand, $G(D_XV,Y)= 0$ by the Koszul formula (\ref{starstar}) and Lemma \ref{3_2}. On the other hand, it follows from (\ref{bb}) and (\ref{starstar}) that 
\begin{eqnarray*}
    2G(D_XV,W) &=& XG(V,W) + VG(X,W) - WG(X,V) + G([X,V],W) - G(V,[X,W]-2P_XW) - G(X,[V,W])\nonumber\\
    &=& XG(V,W) + 2G(V,P_XW)\nonumber\\
    &=& XG(V,W) + 2G(P_XV,W).
\end{eqnarray*}
Since
\begin{equation*}
    XG(V,W) = 2\frac{Xf}{f}G(V,W)
\end{equation*}
by the definition of $G$ and $D$ is torsion free, we obtain
\begin{equation}\label{alpha}
    D_VX =  D_XV =\frac{Xf}{f}V + P_XV.
\end{equation}

Next we consider $D_VW$. Observe that
\begin{equation*}\label{gamma}
    G(D_VW,X) = -G(W,D^*_VX) = -G\left(W,\frac{Xf}{f}V-P_XV\right).
\end{equation*}
Using $Xf = G(\mbox{grad }f,X)$, (\ref{aa}) and (\ref{bb}), we have
\begin{equation*}
    G(D_VW,X) = G\left(-\frac{G(V,W)}{f}\mbox{grad }f + P_VW,X\right).
\end{equation*}
Thus we obtain 
\begin{equation}\label{beta}
    \mbox{Hor }D_VW = -\frac{G(V,W)}{f}\mbox{grad }f + \mbox{Hor }P_VW,
\end{equation}
where $\mbox{Hor}$ denotes the projection to $TB$.
\begin{remark}
    As another way to reach these formulas, we can use the fact that $(D + D^*)/2$ is the Levi-Civita connection and formulas in \cite{oneil}.
    For example,
    \begin{equation*}
        \left(\frac{D + D^*}{2}\right)_XV = \frac{Xf}{f}V 
    \end{equation*}
    implies
    \begin{equation*}
        D_XV = \frac{Xf}{f}V + \frac{D_XV - D^*_XV}{2} = \frac{Xf}{f}V + P_XV.
    \end{equation*}
\end{remark}
\section{Cones}
In this section, we specialize our study of warped products to cones.
We fix our framework and assumptions (including Assumption \ref{assume1}).
\begin{assume}\label{assume2}
    Let $B = \mathbb{R}_{>0}$ with the Euclidean metric $g_B$ such that $g_B(\frac{\partial}{\partial t},\frac{\partial}{\partial t}) = 1$, $f(t) = t$ and $(\widetilde{\nabla},\widetilde{\nabla}^*)$ be dually flat affine connections on $(F,g_F)$.
    Let $D$, $D^*$ be dually flat affine connections on $B\times_f F$ and $G$ be its warped product metric. We assume that $D$ satisfies
    \begin{itemize}
        \item $D_XY \mbox{ is horizontal,  } i.e. , \,D_XY\in\mathcal{L}(B)$ for all $X,Y\in\mathcal{L}(B)$,
        \item $\mbox{Ver }(D_VW) = \mbox{Lift }(\widetilde{\nabla}_VW)$ for all $V,W\in\mathcal{L}(F)$,
    \end{itemize}
    where Ver is the projection to $TF$.
\end{assume}
\begin{remark}
    Denote the curvature with respect to $D$ by $R$ and the curvature with respect to $\widetilde{\nabla}$ by ${}^FR$. Denote their duals by $R^*$ and $^FR^{*}$.
    Note that by Proposition \ref{curvature_prop}, we have $R^* = {}^FR^* = 0$ when $R = {}^FR = 0$.
\end{remark}
In the following arguments in this section, we assume this assumption without mentioning.
We shall study what $R = {}^FR =0$ means and characterize admissible connections on $B$ (Theorem 4.1).

\subsection{Calculations of $G(R(U,V)V,U)$}\label{second}
We first consider in vertical directions.
We are going to calculate the Gauss equation (the relatioin between $R$ and $^FR$) for affine connections on $F$ and $M$ in a similar way to \cite{oneil}.
For $U,V,W,Q\in\mathcal{L}(F)$, it follows from Assumption 4.1 that 
\begin{eqnarray*}
    G(D_UD_VW,Q) &=& G(D_U(\mbox{Ver }D_VW),Q) + G(D_U(\mbox{Hor }D_VW),Q)\nonumber\\
    &=& G(\widetilde{\nabla}_U\widetilde{\nabla}_VW,Q) + \left\{UG(I\hspace{-.1em}I(V,W),Q) - G(I\hspace{-.1em}I(V,W),D^*_UQ)\right\}\nonumber\\
    &=& G(\widetilde{\nabla}_U\widetilde{\nabla}_VW,Q) - G(I\hspace{-.1em}I(V,W),\mbox{Hor }D^*_UQ)\nonumber\\
    &=& G(\widetilde{\nabla}_U\widetilde{\nabla}_VW,Q) -G(I\hspace{-.1em}I(V,W),I\hspace{-.1em}I^*(U,Q)),
\end{eqnarray*}
where $I\hspace{-.1em}I(W,X) := \mbox{Hor }{D}_WX$, which is an affine version of the second fundamental form on $F$. Thus we have
\begin{eqnarray}\label{gauss1}
    G(R(U,V)W,Q) &=& G({}^FR(U,V)W,Q) - G(I\hspace{-.1em}I(V,W),I\hspace{-.1em}I^*(U,Q)) + G(I\hspace{-.1em}I(U,W),I\hspace{-.1em}I^*(V,Q)).
\end{eqnarray}
Recall that ${}^FR = 0$ by Assumption \ref{assume2}. Observe from (\ref{beta}) that 
\begin{eqnarray*}
    G(I\hspace{-.1em}I(V,W),I\hspace{-.1em}I^*(U,Q)) &=& G\left(-\frac{G(V,W)}{f}\mbox{grad}f + \mbox{Hor}(P_VW), -\frac{G(U,Q)}{f}\mbox{grad}f - \mbox{Hor}(P_UQ)\right)\nonumber \\
    &=& \frac{\|\mbox{grad}f\|^2}{f^2}G(V,W)G(U,Q) - G(\mbox{Hor}(P_VW),\mbox{Hor}(P_UQ))\nonumber  \\
    &&\qquad+ \frac{G(V,W)G(\mbox{grad}f,P_UQ)}{f} - \frac{G(U,Q)G(\mbox{grad}f,P_VW)}{f},
\end{eqnarray*}
and similarly
\begin{eqnarray*}
    G(I\hspace{-.1em}I(U,W),I\hspace{-.1em}I^*(V,Q)) &=& \frac{\|\mbox{grad}f\|^2}{f^2}G(U,W)G(V,Q) - G(\mbox{Hor}(P_UW),\mbox{Hor}(P_VQ)) \nonumber \\
    &&\qquad+ \frac{G(U,W)G(\mbox{grad}f,P_VQ)}{f} - \frac{G(V,Q)G(\mbox{grad}f,P_UW)}{f}.
\end{eqnarray*}
Substituting these and letting $W=V$ and $Q = U$,
\begin{eqnarray*}
    G(R(U,V)V,U) &=& - \frac{\|\mbox{grad }f\|^2}{f^2}\left\{G(V,V)G(U,U) - G(U,V)^2\right\}\nonumber\\
    &&\qquad + G(\mbox{Hor}(P_VV),\mbox{Hor}(P_UU)) -G(\mbox{Hor}(P_UV),\mbox{Hor}(P_VU)) - \frac{G(V,V)G(\mbox{grad }f,P_UU)}{f} \nonumber\\
    &&\qquad + \frac{G(U,U)G(\mbox{grad }f,P_VV)}{f}+ \frac{G(U,V)G(\mbox{grad }f,P_VU)}{f} - \frac{G(V,U)G(\mbox{grad }f,P_UV)}{f}\nonumber\\
    &=& - \frac{\|\mbox{grad }f\|^2}{f^2}\left\{G(V,V)G(U,U) - G(U,V)^2\right\}\nonumber\\
    &&\qquad + G(\mbox{Hor}(P_VV),\mbox{Hor}(P_UU)) -G(\mbox{Hor}(P_VU),\mbox{Hor}(P_VU)) - \frac{G(V,V)G(\mbox{grad }f,P_UU)}{f} \nonumber\\
    &&\qquad + \frac{G(U,U)G(\mbox{grad }f,P_VV)}{f},
\end{eqnarray*}
where we used (\ref{aa}).
Recalling $R = 0$ by Assumption \ref{assume2}, for any $U,V\in\mathcal{L}(F)$, we find
\begin{eqnarray*}
    & &\frac{\|\mbox{grad }f\|^2}{f^2}\left\{G(V,V)G(U,U) - G(U,V)^2\right\}- G(\mbox{Hor}(P_VV),\mbox{Hor}(P_UU)) +G(\mbox{Hor}(P_VU),\mbox{Hor}(P_VU)) \nonumber\\
    &&\qquad+ \frac{G(V,V)G(\mbox{grad }f,P_UU)}{f} - \frac{G(U,U)G(\mbox{grad }f,P_VV)}{f} = 0.
\end{eqnarray*} 
Focusing on the symmetric and anti-symmetric parts in $U$ and $V$, and recalling $f(t) = t$, we obtain the following two equations: 
\begin{equation}
    \label{a}
    G(V,V)G(\mbox{grad }f,P_UU) = G(U,U)G(\mbox{grad }f,P_VV),
\end{equation}
\begin{equation}
    \label{b}
    \frac{1}{t^2}\left\{G(V,V)G(U,U) - G(U,V)^2\right\}- G(\mbox{Hor}(P_VV),\mbox{Hor}(P_UU)) +G(\mbox{Hor}(P_VU),\mbox{Hor}(P_VU)) = 0.
\end{equation}
To characterize admissible connections on $B$, we prepare some lemmas.
\begin{lemma}\label{lem4_1}
    We have
    \begin{equation*}
        {\rm{Hor}}\left(P_{\frac{U}{\|U\|}}\frac{U}{\|U\|}\right) = {\rm{Hor}}\left(P_{\frac{V}{\|V\|}}\frac{V}{\|V\|}\right)
    \end{equation*}
    for any $U,V\in\mathcal{L}(F)$.
\end{lemma}
\begin{proof}
    It follows from (\ref{a}) that
    \begin{equation*}
        G\left(\mbox{grad }f,\mbox{Hor}P_{\frac{U}{\|U\|}}\frac{U}{\|U\|}\right) = G\left(\mbox{grad }f,\mbox{Hor}P_{\frac{V}{\|V\|}}\frac{V}{\|V\|}\right)
    \end{equation*}
    and we get
    \begin{equation*}
        \mbox{Hor}\left(P_{\frac{U}{\|U\|}}\frac{U}{\|U\|}\right) = \mbox{Hor}\left(P_{\frac{V}{\|V\|}}\frac{V}{\|V\|}\right).
    \end{equation*}
\end{proof}

Hereafter, fix an arbitrary $x\in F$. Let $(\xi_i)$ be normal coordinates around $x$ on $F$ and denote $\partial_i = \frac{\partial}{\partial\xi_i}$.
\begin{lemma}\label{lem4_2}
    We have
    \begin{equation*}
        ({\rm{Hor}}P_{\partial_i}\partial_i)_{(x,t)} = \left(t\frac{\partial}{\partial t}\right) \mbox{ or } \left(- t\frac{\partial}{\partial t}\right).
    \end{equation*}
\end{lemma}
\begin{proof}
    Put $U = \partial_i$ and $V = \partial_j(i\neq j)$. Then Lemma \ref{lem4_1} implies
    \begin{equation*}
        \frac{\mbox{Hor } P_{(U + V)}(U + V)}{\|U + V\|^2} = \frac{\mbox{Hor }P_UU}{\|U\|^2} = \frac{\mbox{Hor }P_VV}{\|V\|^2}.
    \end{equation*}
Note that, for all $t > 0$,
    \begin{equation*}
        \|U + V\|^2_{(x,t)} = \|U\|^2_{(x,t)} + \|V\|^2_{(x,t)} = 2t^2.
    \end{equation*}
This yields 
    \begin{equation*}
        \left\{\mbox{Hor}P_{(U + V)}(U + V)\right\}_{(x,t)} = 2\left(\mbox{Hor}P_UU\right)_{(x,t)} = 2(\mbox{Hor }P_VV)_{(x,t)},
    \end{equation*}
    and hence
    \begin{equation}\label{delta}
        (\mbox{Hor}P_UV)_{(x,t)} = 0.
    \end{equation}
Combining this with (\ref{b}), we have for all $t$,

    \begin{equation*}
        t^2= G_{(x,t)}(\mbox{Hor}P_VV,\mbox{Hor}P_UU) = G_{(x,t)}(\mbox{Hor}P_UU,\mbox{Hor}P_UU).
    \end{equation*}
    This proves the claim.
\end{proof}

\subsection{Calculations of $G(R(V,X)X,V)$}\label{coneconnection}
Now, we put $X = \frac{\partial}{\partial t}$ and define $k$ by $D_XX = k(t)\frac{\partial}{\partial t}$. 
\begin{lemma}\label{lem4_3}
    Let $\partial_i = \frac{\partial}{\partial\xi_i}$ as in Lemma \ref{lem4_2}.
    We have 
    \begin{equation*}
        (P_X\partial_i)_{(x,t)} = \left(\frac{1}{t}\partial_i\right)_{(x,t)} \mbox{ or } \left(- \frac{1}{t}\partial_i\right)_{(x,t)}
    \end{equation*}
    for all $t > 0$.
\end{lemma}
\begin{proof}
    Put $V = \partial_i$. Recall that $(\mbox{Hor }P_VV)_{(x,t)} = t\frac{\partial}{\partial t} \mbox{ or } \left(-t\frac{\partial}{\partial t}\right)$ by Lemma \ref{lem4_2}.
    First, we consider the case $(\mbox{Hor}P_VV)_{(x,t)} = t\frac{\partial}{\partial t}$.
We deduce from (\ref{alpha}) that 
    \begin{eqnarray*}
        G_{(x,t)}(D^*_VD_X^*X,V) &=& G\left(-k(t)\left(\frac{1}{t}V-P_XV\right),V\right)\nonumber\\
        &=& -k(t)t + k(t)t\nonumber\\
        &=& 0
    \end{eqnarray*}
    for all $t > 0$, where the second equality follows since $G(P_XV,V) = G(P_VX,V) = G(X,P_VV) = t$ by (\ref{aa}) and (\ref{bb}).
    We similarly find from (\ref{alpha}) that
    \begin{eqnarray*}
        G_{(x,t)}(D^*_XD^*_VX,V) &=& XG\left(\frac{1}{t}V - P_XV,V\right) - G\left(\frac{V}{t}- P_XV,D_XV\right)\nonumber\\
        &=& \frac{\partial}{\partial t}\left(\frac{1}{t}t^2 - t\right) - G\left(\frac{V}{t} - P_XV,\frac{V}{t} + P_XV\right)\nonumber\\
        &=& -G\left(\frac{V}{t},\frac{V}{t}\right) + G(P_XV,P_XV)\nonumber\\
        &=& -1 + G(P_XV,P_XV).
    \end{eqnarray*}
Therefore we obtain for all $t > 0$, since $R^* = 0$,
    \begin{equation*}
        G(P_XV,P_XV) = 1.
    \end{equation*}

    Next we consider the case $(\mbox{Hor}P_VV)_{(x,t)} = -t\frac{\partial}{\partial t}$.
    We have
    \begin{eqnarray*}
        G_{(x,t)}(D_V(D_XX),V) &=& G\left(k(t)\left(\frac{1}{t}V+P_XV\right),V\right)\nonumber\\
        &=& tk(t) - tk(t)\nonumber\\
        &=& 0
    \end{eqnarray*}
    and
    \begin{eqnarray*}
        G_{(x,t)}(D_X(D_VX),V) &=& G\left(D_X\left(\frac{1}{t}V + P_XV\right),V\right)\nonumber\\
        &=& XG\left(\frac{1}{t}V + P_XV,V\right) - G\left(\frac{V}{t} + P_XV,\frac{V}{t} - P_XV\right)\nonumber\\
        &=&- 1 + G(P_XV,P_XV).
    \end{eqnarray*}
Since $R = 0$, we have $G_{(x,t)}(P_XV,P_XV) = 1$.

    For $U = \partial_j(i\neq j)$, using $(\mbox{Hor}P_VU)_{(x,t)} = 0$ in (\ref{delta}), (\ref{aa}) and (\ref{bb}), we have
    \begin{equation*}
        G_{(x,t)}(P_XV,U)  = G_{(x,t)}(P_VX,U) = G_{(x,t)}(X,P_VU) = 0.
    \end{equation*}
Moreover $G(P_XV,X) = G(V,P_XX) = 0$. Therefore $(P_XV)_{(x,t)}$ and $V_{(x,t)}$ are linearly dependent for all $t > 0$, which proves the claim.
\end{proof}
The next result is the aim of this section, which is a characterization of connections on the line $B$.
\begin{theorem}\label{mainthm}
    Under Assumption \ref{assume2}, we have
    \begin{equation*}
        k(t) = \frac{1}{t} \mbox{ or } \left(-\frac{1}{t}\right).
    \end{equation*}
\end{theorem}
\begin{proof}
    Put $V = \partial_i$. 
    When $(\mbox{Hor}P_VV)_{(x,t)} = -t\frac{\partial}{\partial t}$, we have
    \begin{equation*}
        G_{(x,t)}(P_XV,V) = G_{(x,t)}(P_VX,V) = G_{(x,t)}(X,P_VV) = -t.
    \end{equation*}
Combining this with Lemma \ref{lem4_3}, we find 
    \begin{equation*}
        (P_VX)_{(x,t)} = -\frac{1}{t}V.
    \end{equation*}
We similarly find that $P_VX = \frac{1}{t}V$ if $\mbox{Hor }P_VV = t\frac{\partial}{\partial t}$.
Hence, we need to consider only the following two cases.

    First, we consider the case $(\mbox{Hor}P_VV)_{(x,t)} = t\frac{\partial}{\partial t}$ and $(P_VX)_{(x,t)} = \frac{1}{t}V$. We have
    \begin{eqnarray*}
        G(D_V(D_XX),V) &=& G\left(k(t)\left(\frac{1}{t}V + P_V\frac{\partial}{\partial t}\right),V\right) =  2tk(t),
    \end{eqnarray*}
and
    \begin{eqnarray*}
        G(D_X(D_VX),V) &=& XG(D_VX,V) - G(D_VX,D^*_XV)\nonumber\\
        &=& \frac{\partial}{\partial t}\left\{G\left(\frac{1}{t}V + P_XV,V\right)\right\}-G\left(\frac{V}{t} + P_XV,\frac{V}{t}-P_XV\right)\nonumber\\
        &=& \frac{\partial}{\partial t}(t + t) - 1 + 1 = 2.
    \end{eqnarray*}
Hence by $R = 0$, we obtain $k(t) = \frac{1}{t}$.

Next, we consider the case $(\mbox{Hor}P_VV)_{(x,t)} = -t\frac{\partial}{\partial t}$ and $(P_VX)_{(x,t)} = -\frac{1}{t}V$.
We similarly have 
    \begin{eqnarray*}
        G(D_V^*D_X^*X,V) &=& -k(t)G\left(\frac{1}{t}V - P_V\frac{\partial}{\partial t},V\right) = -2tk(t),
    \end{eqnarray*}
and
    \begin{eqnarray*}
        G(D_X^*D_V^*X,V) &=& \frac{\partial}{\partial t}(t + t) - G\left(\frac{V}{t} - P_XV,\frac{V}{t} + P_XV\right) = 2-1 + 1 = 2.
    \end{eqnarray*}
    Therefore $k(t) = -\frac{1}{t}$.
\end{proof}
From the above proof, we obtain that $\mbox{Hor }(P_{\partial_i}\partial_i) = t\frac{\partial}{\partial t}$ and $P_{\partial_i}X = \frac{1}{t}\partial_i$ if $k(t) = \frac{1}{t}$, and 
that $\mbox{Hor }(P_{\partial_i}\partial_i) = -t\frac{\partial}{\partial t}$ and $P_{\partial_i}X = -\frac{1}{t}\partial_i$ if $k(t) = -\frac{1}{t}$.

\subsection{Example 1: Denormalization}
Here we consider the denormalization (recall Subsection \ref{infogeo}) as an example of warped product with affine connections.
Since we can prove the isometry to a warped product in the same way as in Subsection {\ref{bkmcone}}, we omit a
detailed proof and see an explicit expression of an isometry between the warped product and the denormalization.

Let $\widetilde{\mathcal{S}}$ be the set of positive finite measures on a finite set $\mathcal{X}$. We define a map $h$ as 
\begin{eqnarray*}
    h:\mathbb{R}_{>0}\times \mathcal{S} &\rightarrow& \widetilde{\mathcal{S}}\nonumber\\
    (t,p) &\longmapsto& t^2p/4.
\end{eqnarray*}
We pull back $\tilde{g}$ on $\widetilde{\mathcal{S}}$ in (\ref{netric_denormalization}) by $h$ and define the induced metric $G$ on $\mathbb{R}_{>0}\times \mathcal{S}$. This $(\mathbb{R}_{>0}\times \mathcal{S},G)$ is a warped product and 
$c(t) := t^2p/4$ is a line of constant speed 1.

Let $\{\xi_1,\ldots,\xi_n\}$ be a coordinate system of $\mathcal{S}$. We adopt $\{\tau,\xi_1,\ldots,\xi_n\}$ as a coordinate system of $\tilde{\mathcal{S}}$ and 
denote its natural basis by $\tilde{\partial}_i = \frac{\partial}{\partial\xi_i}$ and $\tilde{\partial}_\tau = \frac{\partial}{\partial \tau}$.
For a vector field $X = X^i\partial_i\in \mathfrak{X}(\mathcal{S})$, we define $\tilde{X} := X^i\tilde{\partial_i}\in \mathfrak{X}(\widetilde{\mathcal{S}})$.

We can set affine connections on the denormalization as in Subsection \ref{infogeo}. In \cite{amari}, it is expressed as follows. For $X,Y\in \mathfrak{X}(\mathcal{S})$,
\begin{eqnarray*}
    \widetilde{\nabla}^{(\alpha)}_{\widetilde{X}}\widetilde{Y} &=& \widetilde{(\nabla^{(\alpha)}_XY)} - \frac{1 + \alpha}{2}\langle\widetilde{X},\widetilde{Y}\rangle\tilde{\partial}_\tau,\nonumber\\
    \widetilde{\nabla}^{(\alpha)}_{\tilde{\partial}_\tau}\widetilde{X} &=& \widetilde{\nabla}^{(\alpha)}_{\widetilde{X}}\tilde{\partial}_\tau = \frac{1-\alpha}{2}\frac{1}{\tau}\widetilde{X},\nonumber\\
    \widetilde{\nabla}^{(\alpha)}_{\tilde{\partial}_\tau}\tilde{\partial}_\tau &=& -\frac{1 + \alpha}{2}\frac{1}{\tau}\tilde{\partial}_\tau.
\end{eqnarray*}
Also, the metric $\widetilde{g}$ is expressed as
\begin{equation*}
    \widetilde{g}_{ij} = \tau g_{ij},\quad \widetilde{g}_{i\tau} = 0,\quad\widetilde{g}_{\tau\tau} = \frac{1}{\tau}.
\end{equation*}

We can check that this connection satisfies Assumption \ref{assume2} by direct calculations, that is to say, the $\alpha$-connection on the denormalization is compatible with the warped product structure and their curvatures vanish at $\alpha = \pm{1}$.
Let us see that the results we obtained in Subsections \ref{second} and \ref{coneconnection} are also obtained in this situation.
We set $\tau = t^2/4$. Note that $\|\partial_\tau\| := \sqrt{\tilde{g}(\partial_\tau,\partial_\tau)}= 1/\sqrt{\tau}$. We have, by omitting the tilde for simplicity, 
\begin{eqnarray*}
    D_{\frac{\partial_\tau}{\|\partial_\tau\|}}X &=& \frac{1}{\|\partial_\tau\|}\frac{1-\alpha}{2}\frac{1}{\tau} X = \frac{1-\alpha}{2}\frac{1}{\sqrt{\tau}}X = \frac{1-\alpha}{t}X,\nonumber\\
    P_X\frac{\partial_\tau}{\|\partial_\tau\|} &=& P_{\frac{\partial_\tau}{\|\partial_\tau\|}}X = \frac{1}{2}\left(\frac{1-\alpha}{t}-\frac{1 + \alpha}{t}\right)X = -\frac{\alpha}{t}X,\nonumber\\
    D_{\frac{\partial_\tau}{\|\partial_\tau\|}}\frac{\partial_\tau}{\|\partial_\tau\|} &=& \sqrt{\tau}\left\{(\partial_\tau\sqrt{\tau})\partial_\tau + \frac{1}{\|\partial_\tau\|}\left(-\frac{1+\alpha}{2}\frac{1}{\tau}\partial_\tau\right)\right\} = -\frac{\alpha}{2}\partial_\tau = -\frac{\alpha}{t}\frac{\partial_\tau}{\|\partial_\tau\|},
\end{eqnarray*}
where $D = \widetilde{\nabla}^{(\alpha)}$.
When $\alpha = \pm{1}$, these equations are compatible with the connections in Subsection \ref{coneconnection}.

\subsection{Example 2: BKM cone}\label{bkmcone}
Next, we consider $\mathbb{P}(n)$ equipped with the extended BKM metric (recall Subsection 2.2), which we call the BKM cone.
We first show that $\mathbb{P}(n)$ with the extended monotone metric (not only the BKM metric) has a warped product structure. 

\begin{proposition}
   $\mathbb{P}(n)$ equipped with the extended monotone metric is a warped product. Precisely, there exists an isometry as follows:
    \begin{eqnarray*}
        \mathbb{R}_{>0}\times_{l(t) = t}\mathcal{D} &\rightarrow& (\mathbb{P}(n),g^f)\nonumber\\
        (t,\rho) &\mapsto& \frac{t^2\rho}{4},
    \end{eqnarray*}
    where $g^f$ is an arbitrary monotone metric.
\end{proposition}
\begin{proof}
    For simplicity, we calculate $2\times 2$ matrices as in \cite{dit}. The following argument can be easily extended to the $n\times n$ case. 
    For an arbitrary $\rho \in\mathcal{D}$, there exists a unitary matrix $U$ such that $U\rho U^* = \rho_0$, where 
    $\rho_0 = \mbox{diag}[x,y]$ for some $x,y\in\mathbb{R}$.
    We set
    \begin{equation*}
        X_1 = \begin{pmatrix}
            2 & 0\\
            0 & 0
        \end{pmatrix},\quad X_2 = \begin{pmatrix}
            0 & 0\\
            0 & 2
        \end{pmatrix},\quad 
        X_3 = \begin{pmatrix}
            0 & 1\\
            1 & 0
        \end{pmatrix},\quad 
        X_4 = \begin{pmatrix}
            0 & i\\
            -i & 0
        \end{pmatrix}.
    \end{equation*}
    These $X_1,X_2,X_3,X_4$ form an orthogonal basis of every tangent space of $(\mathbb{P}(2),g^f)$.
 Let us calculate the length of these vectors at $\rho_0$:
    \begin{eqnarray*}
        g_{\rho_0}^f(X_1,X_1) &=& \frac{1}{(2\pi i)^2}\mathrm{Tr}\oint\oint c(\xi,\eta)\begin{pmatrix}
            2 & 0\\
            0 & 0
        \end{pmatrix}\begin{pmatrix}
            \frac{1}{\xi-x} & 0\\
            0 & \frac{1}{\xi-y}
        \end{pmatrix}\begin{pmatrix}
            2 & 0\\
            0 & 0
        \end{pmatrix}\begin{pmatrix}
            \frac{1}{\eta-x} & 0\\
            0 & \frac{1}{\eta-y}
        \end{pmatrix}\, d\xi d\eta\nonumber\\
        &=& \frac{1}{(2\pi i)^2}\mathrm{Tr}\oint\oint c(\xi,\eta)\begin{pmatrix}
            \frac{4}{(\xi - x)(\eta - x)} & 0\\
            0 & 0
        \end{pmatrix}\, d\xi d\eta\nonumber\\
        &=& 4c(x,x),\nonumber\\
        g^f_{\rho_0}(X_2,X_2) &=& 4c(y,y),\nonumber\\
        g_{\rho_0}^f(X_3,X_3) &=& \frac{1}{(2\pi i)^2}\mathrm{Tr}\oint\oint c(\xi,\eta)\begin{pmatrix}
            \frac{1}{(\xi-y)(\eta-x)} & 0\\
            0 & \frac{1}{(\xi-x)(\eta-y)}
        \end{pmatrix}\, d\xi d\eta\nonumber\\
        &=& 2c(x,y),\nonumber\\
        g_{\rho_0}^f(X_4,X_4) &=& \frac{1}{(2\pi i)^2}\mathrm{Tr}\oint\oint c(\xi,\eta)\begin{pmatrix}
            \frac{1}{(\xi-y)(\eta-x)} & 0\\
            0 & \frac{1}{(\xi-x)(\eta-y)}
        \end{pmatrix}\, d\xi d\eta\nonumber\\
        &=& 2c(x,y).
    \end{eqnarray*}
    These calculations show that, for any $k> 0$ and any tangent vectors $X$ and $Y$, we have 
    \begin{equation*}
        g^f_{k\rho}(X,Y) = g^f_{k\rho_0}(UXU^*,UYU^*) = \frac{1}{k}g^f_{\rho_0}(UXU^*,UYU^*) = \frac{1}{k}g^f_\rho(X,Y),
    \end{equation*}
    where we used the fact that $g^f_{U\rho U^*}(UXU^*,UYU^*) = g^f_\rho(X,Y)$.
    Thus, we obtain
    \begin{equation}\label{mono1}
        g^f_{k\rho}(kX,kX) = kg^f_\rho(X,X).
    \end{equation}
We define $h$ as 
    \begin{eqnarray*}
        h:\mathbb{R}_{>0}\times\mathcal{D} &\rightarrow& \mathbb{P}(2)\nonumber\\
        (t,\rho) &\mapsto& \frac{t^2\rho}{4}.
    \end{eqnarray*}
    We pull back $g^f$ on $\mathbb{P}(n)$ by $h$ and define $G$ on $\mathbb{R}_{>0}\times \mathcal{D}$.
    We show that $G$ is a warped product metric on $\mathbb{R}_{>0}\times\mathcal{D}$.
    We consider the lines
    \begin{equation*}
        \gamma(t) := \frac{t^2\rho}{4},\quad \gamma_0(t) := \frac{t^2\rho_0}{4}.
    \end{equation*}
    Then we find 
    \begin{eqnarray*}
        G_{(t,\rho)}\left(\frac{\partial}{\partial t},\frac{\partial}{\partial t}\right) &=& g_{\gamma(t)}^f(\gamma'(t),\gamma'(t))\nonumber\\
        &=& g^f_{U\gamma(t)U^*}(U\gamma'(t)U^*,U\gamma'(t)U^*)\nonumber\\
        &=&g^f_{\gamma_0(t)}(\gamma_0'(t),\gamma_0'(t))\nonumber\\
         &=& g_{\gamma_0(t)}^f\left(\begin{pmatrix}
            tx/2 & 0\\
            0 & 0
        \end{pmatrix},\begin{pmatrix}
            tx/2 & 0\\
            0 & 0
        \end{pmatrix}\right)+g_{\gamma_0(t)}^f\left(\begin{pmatrix}
            0 & 0\\
            0 & ty/2
        \end{pmatrix},\begin{pmatrix}
            0 & 0\\
            0 & ty/2
        \end{pmatrix}\right)\nonumber\\
        &=& \left(\frac{tx}{2}\right)^2 \frac{4}{t^2x} + \left(\frac{ty}{2}\right)^2 \frac{4}{t^2y}\nonumber\\
        &=& \mathrm{Tr}\rho.
    \end{eqnarray*}
    Hereafter, we assume $x + y = 1$. We now get
    \begin{equation}\label{warp1}
        G_{(t,\rho)}\left(\frac{\partial}{\partial t},\frac{\partial}{\partial t}\right) = 1.
    \end{equation}

    Let us calculate $dh$. Since $h$ is expressed by the natural coordinates of $\mathbb{R}_{>0}\times \mathcal{D}$ and $\mathbb{P}(2)$ as 
    \begin{equation*}
        \left(t,\begin{pmatrix}
            x & z + iw\\
            z - iw & 1-x
        \end{pmatrix}\right)\longmapsto\frac{t^2}{4}\begin{pmatrix}
            x & z + iw\\
            z - iw & 1-x
        \end{pmatrix} = \begin{pmatrix}
            a & b + ic\\
            b - ic & d
        \end{pmatrix},
    \end{equation*}
    we have
    \begin{equation*}
        (Jh)_\rho = \begin{pmatrix}
            tx/2 & t^2/4  & 0 & 0 \\
            tz/2 & 0  & t^2/4 & 0 \\
            tw/2 & 0  & 0 & t^2/4 \\
            t(1-x)/2 & -t^2/4  & 0 & 0 \\
        \end{pmatrix}.
    \end{equation*}
The pull-back metric $G$ satisfies
    \begin{eqnarray*}
        G\left(\left(\frac{\partial}{\partial z}\right)_{(t,\rho)},\left(\frac{\partial}{\partial z}\right)_{(t,\rho)}\right) = g^f\left(\frac{t^2}{4}\left(\frac{\partial}{\partial b}\right)_{\frac{t^2}{4}\rho},\frac{t^2}{4}\left(\frac{\partial}{\partial b}\right)_{\frac{t^2}{4}\rho}\right),
    \end{eqnarray*}
    where
    \begin{equation*}
        \left(\frac{\partial}{\partial b}\right)_{\frac{t^2}{4}\rho} := \begin{pmatrix}
            0 & 1\\
            1 & 0
        \end{pmatrix}.
    \end{equation*}
    Using (\ref{mono1}), we obtain
    \begin{eqnarray*}
        g^f\left(\frac{t^2}{4}\left(\frac{\partial}{\partial b}\right)_{\frac{t^2}{4}\rho},\frac{t^2}{4}\left(\frac{\partial}{\partial b}\right)_{\frac{t^2}{4}\rho}\right) &=& t^2g^f_{\rho/4}\left(\frac{1}{4}\begin{pmatrix}
            0 & 1\\
            1 & 0
        \end{pmatrix}
            ,\frac{1}{4}\begin{pmatrix}
                0 & 1\\
                1 & 0
            \end{pmatrix}\right)\nonumber\\
        &=& t^2G\left(\left(\frac{\partial}{\partial z}\right)_{(1,\rho)},\left(\frac{\partial}{\partial z}\right)_{(1,\rho)}\right).
    \end{eqnarray*}
    Hence,
    \begin{equation}\label{warp2}
        G_{(t,\rho)}\left(\frac{\partial}{\partial z},\frac{\partial}{\partial z}\right) = t^2G_{(1,\rho)}\left(\frac{\partial}{\partial z},\frac{\partial}{\partial z}\right).
    \end{equation}
    The same equation holds for $\frac{\partial}{\partial x}$ and $\frac{\partial}{\partial w}$.
    Combining this with \eqref{warp1}, we see that $G$ is a warped product metric with the warping function $l(t) = t$.
\end{proof}

    For $n=2$, if we take trivial coordinates of $\mathcal{D}$ such as 
    \begin{equation*}
        \rho(x,y,z) = \begin{pmatrix}
            x & y + iz\\
            y - iz & 1-x
        \end{pmatrix},
    \end{equation*}
    the mixture connection is an affine connection, for which $\{x,y,z\}$ is affine coordinates.
    For example, we have the following calculation for the mixture connection ${\nabla}^{(m)}$. Set $X = \frac{\partial}{\partial x}$, $Y = \phi\frac{\partial}{\partial y}$, for an arbitrary function $\phi$ on $\mathcal{D}$. Then we have
    \begin{equation*}
        (\nabla^{(m)}_XY)_\rho = \left\{\nabla^{(m)}_{\frac{\partial}{\partial x}}\left(\phi\frac{\partial}{\partial y}\right) \right\}_{\rho}= \left(\frac{\partial \phi}{\partial x}\frac{\partial}{\partial y}\right)_{\rho} = \frac{\partial \phi}{\partial x}\begin{pmatrix}
            0 & 1\\
            1 & 0
        \end{pmatrix} = \frac{\partial}{\partial x}\left(\phi\begin{pmatrix}
            0 & 1\\
            1 & 0
        \end{pmatrix}\right) = X(Y\rho).
    \end{equation*}
    Similarly, if we take the coordinates of $\mathbb{P}(2)$ such as 
    \begin{equation*}
        \widetilde{\rho} = \begin{pmatrix}
            \alpha & \beta + i\gamma\\
            \beta - i \gamma & \zeta
        \end{pmatrix},
    \end{equation*}
    then the connection $D$ whose affine coordinate system is $\{\alpha,\beta,\gamma,\zeta\}$ satisfies $(D_XY)_{\widetilde{\rho}} = X(Y\widetilde{\rho})$.

 For another coordinates such as 
\begin{equation*}
    \widetilde{\rho} = \tau\begin{pmatrix}
        x & y + iz\\
        y-iz & 1-x
    \end{pmatrix},
\end{equation*}
we have $D_{\partial_\tau}\partial_\tau = 0$ and $D_{\partial_x}\partial_x = D_{\partial_y}\partial_y  = D_{\partial_z}\partial_z = 0$. Combining this with $\widetilde{\nabla}_{\partial_x}\partial_x = \widetilde{\nabla}_{\partial_y}\partial_y = \widetilde{\nabla}_{\partial_z}\partial_z = 0$, we see that the connection $D$ defined above satisfies Assumption \ref{assume2}.

\begin{remark}
    In \cite{grasseli}, quantum $\alpha$-connections on the set of positive definite matrices and their dually flatness are studied. Also for the quantum $\alpha$-connections, we can check the same compatibility between connections and the warped product structure as that of the classical denormalization, which we studied in Subsection 4.3.
\end{remark}

\section{Connections on the Takano Gaussian space}
In this section, we consider the Takano Gaussian space (recall Subsection 2.3) and show an analogue to Theorem 4.1.

Let $(B,g_B),(F,g_F)$ be Riemannian manifolds and $\nabla^F$ denotes the Levi-Civita connection of $F$.
We furnish $M := B \times F$ with a metric $G$ such that
\begin{equation*}
    G := f^2 g_B + b^2 g_F,
\end{equation*}
where $f$, $b$ are positive functions on $B$. This is the same situation as in the Takano Gaussian space.
Denote by $\nabla$ the Levi-Civita connection on $(M,G)$, and by $\mathcal{L}(F)$, $\mathcal{L}(B)$ the sets of lifts of tangent vector fields of $F$ to $M$, $B$ to $M$, respectively.
Simple calculations show that $\nabla_XY$ is horizontal for any $X,Y\in\mathcal{L}(B)$ and $\mbox{Ver}\nabla_VW = \mbox{Lift}(\nabla^F_VW)$ for any $V,W\in\mathcal{L}(F)$.

From now on, we set $M:=L^{n+1}$, $F:= \{(m_1,\ldots,m_n) | m_i\in\mathbb{R}\}$ and $B := \{\sigma\in\mathbb{R}_{>0}\}$. Let $G$ be the Fisher metric on $M$. Let $D$ be an arbitrary affine connection on $M$. We define the affine connection $\widetilde{\nabla}^F$ on the fiber space $F$ by the natural projection of $D$. 

We fix our framework.
\begin{assume}\label{assume_5}
We assume that $D$ satisfies
    \begin{itemize}
        \item $D_XY \mbox{ is horizontal,  } i.e. ,\, D_XY\in\mathcal{L}(B)$ for any $X,Y\in\mathcal{L}(B)$,
        \item $\mbox{Ver }(D_VW) = \mbox{Lift }(\widetilde{\nabla}^F_VW)$ for any $V,W\in\mathcal{L}(F)$.
    \end{itemize}
    Let $R$ be the curvature with respect to $D$, $^FR$ be the curvature with respect to $\widetilde{\nabla}^F$, and $R^*$ and $^FR^{*}$ be their duals.
    We also assume $R = {}^FR =0$.
\end{assume}
\begin{remark}
    If we take an $\alpha$-connection of the Takano Gaussian space as $D$, we see that $(F,\widetilde{\nabla}^F)$ is dually flat from the expression of the Christoffel symbols of the Takano Gaussian space in Subsection \ref{takano_gauss}.
    \end{remark}
In the following arguments in this section, we assume this assumption without mentioning.
As we saw in Subsection 4.1, the following equations hold:
\begin{equation}
    \label{c}
    G(V,V)G(\mbox{grad }b,P_UU) = G(U,U)G(\mbox{grad }b,P_VV),
\end{equation}
\begin{equation}
    \label{d}
    \frac{\|\mbox{grad }b\|^2}{b^2}\left\{G(V,V)G(U,U) - G(U,V)^2\right\}- G(\mbox{Hor}(P_VV),\mbox{Hor}(P_UU)) +G(\mbox{Hor}(P_VU),\mbox{Hor}(P_VU)) = 0,
\end{equation}
where $b(\sigma) = \frac{\sqrt{2n}}{\sigma}$.
In the following arguments, we set $X =\partial_\sigma$, $U= \partial_i,V = \partial_j$, where $\partial_\sigma = \frac{\partial}{\partial\sigma}$ and $\partial_i = \frac{\partial}{\partial m_i}$.
To characterize the connections on the line, we prepare some lemmas.

\begin{lemma}\label{prop5_4}
    We have
    \begin{equation*}
        ({\rm{Hor}}P_UU)_{(x,\sigma)} = {\frac{1}{2n\sigma}}\partial_\sigma \mbox{ or } \left(- {\frac{1}{2n\sigma}}\partial_\sigma\right).
    \end{equation*}
\end{lemma}
\begin{proof}
    In the same way as Lemma \ref{lem4_1}, we have 
    \begin{equation}\label{lem5_1}
        {\rm{Hor}}P_UU = {\rm{Hor}}P_VV.
    \end{equation}
    Applying (\ref{c}) to $U + V$ and $V$, we obtain
    \begin{equation*}
        G(U + V,U + V)G(\mbox{grad }b,P_VV) = G(V,V)G(\mbox{grad }b,P_{(U + V)}( U + V)).
    \end{equation*}
    Substituting $G(U + V ,U + V) = 2/\sigma^2$ and $G(V,V) = 1/\sigma^2$ to the equation above, we get
    \begin{equation*}
        2\mbox{Hor }P_VV = \mbox{Hor }P_{(U + V)}(U + V).
    \end{equation*}
    Hence, 
    \begin{equation*}
        \mbox{Hor }(P_UU + P_VV + 2 P_UV) = 2\mbox{Hor }P_VV.
    \end{equation*}
    Together with (\ref{lem5_1}), we get
    \begin{equation*}
        (\mbox{Hor}P_UV)_{(x,\sigma)} = 0.
    \end{equation*}
Combining this with (\ref{d}) and 
\begin{equation*}
    \frac{\|\mbox{grad }b\|^2}{b^2} = \frac{G(\mbox{grad }b,\mbox{grad }b)}{b^2} = \frac{G\left(G^{\sigma\sigma}\partial_\sigma\left(\frac{1}{\sigma}\right)\partial_\sigma,G^{\sigma\sigma}\partial_\sigma\left(\frac{1}{\sigma}\right)\partial_\sigma\right)}{\left(\frac{1}{\sigma^2}\right)} = \left(\frac{\sigma}{2n}\right)^2G(\partial_\sigma,\partial_\sigma) = \frac{1}{2n},
\end{equation*}
we have, for all $t > 0$,
    \begin{equation*}
        \frac{1}{2n}\frac{1}{\sigma^2}\frac{1}{\sigma^2} =G_{(x,\sigma)}(\mbox{Hor}P_VV,\mbox{Hor}P_UU)=  G_{(x,\sigma)}(\mbox{Hor}P_UU,\mbox{Hor}P_UU),
    \end{equation*}
    which proves the claim.
\end{proof}

We define $k,l$ by $D_XX = k(\sigma)\frac{\partial}{\partial\sigma}$ and $D^*_XX = l(\sigma)\frac{\partial}{\partial\sigma}$.
Combining $G(\partial_\sigma,\partial_\sigma) = \frac{2n}{\sigma^2}$ with 
   $ \partial_\sigma G(\partial_\sigma,\partial_\sigma) = G(D_{\partial_\sigma}\partial_\sigma,\partial_\sigma) + G(\partial_\sigma,D^*_{\partial_\sigma}\partial_\sigma)$,
we obtain
\begin{equation*}
    -\frac{4n}{\sigma^3} = \frac{2n}{\sigma^2}(k(\sigma) + l(\sigma)).
\end{equation*}
Hence,
\begin{equation}\label{star5}
    -\frac{2}{\sigma} = k(\sigma) + l(\sigma)
\end{equation}
holds.
\begin{theorem}\label{takano_2}
        Under Assumption \ref{assume_5}, the connection on the line is
        \begin{equation*}
            D_{\partial_\sigma}\partial_\sigma = \frac{1}{\sigma}\partial_\sigma \mbox{ or } D_{\partial_\sigma}\partial_\sigma = -\frac{3}{\sigma}\partial_\sigma.
        \end{equation*}
\end{theorem}
\begin{proof}
    We only check the case of
    $\mbox{Hor}P_VV = \frac{1}{2n\sigma}\partial_\sigma$, because the other case of $\mbox{Hor}P_VV = -\frac{1}{2n\sigma}$ follows from the completely same argument.
    According to the O'Neill formula (\ref{alpha}) for affine connections, we have
    \begin{equation*}
        D_VX = D_XV = \frac{\partial_\sigma b}{b}V + P_VX = -\frac{1}{\sigma}V + P_XV.
    \end{equation*}
    Using this, we calculate $G(R^*(V,X)X,V)$. We have
    \begin{eqnarray*}
        G(D^*_VD^*_XX,V) &=& G\left(l(\sigma)\left\{-\frac{1}{\sigma}V-P_XV\right\},V\right)\nonumber\\
        &=& l(\sigma)\left\{-\frac{1}{\sigma}G(V,V) - G(X,P_VV)\right\}\nonumber\\
        &=& l(\sigma)\left(-\frac{1}{\sigma}\frac{1}{\sigma^2} - G\left(\partial_\sigma,\frac{1}{2n\sigma}\partial_\sigma\right)\right)\nonumber\\
        &=& l(\sigma)\left(-\frac{1}{\sigma^3} - \frac{1}{2n\sigma}\frac{2n}{\sigma^2}\right)\nonumber\\
        &=& l(\sigma)\left(-\frac{2}{\sigma^3}\right),
    \end{eqnarray*}
    and
    \begin{eqnarray*}
        G(D^*_XD^*_VX,V) &=& XG(D^*_VX,V) - G(D^*_VX,D_XV)\nonumber\\
        &=& \partial_\sigma G\left(-\frac{1}{\sigma}V-P_XV,V\right) - G\left(\frac{V}{\sigma},\frac{V}{\sigma}\right) + G(P_XV,P_XV)\nonumber\\
        &=& \partial_\sigma\left(-\frac{1}{\sigma}G(V,V)-G(X,P_VV)\right) - \frac{1}{\sigma^2}G(V,V) + G(P_XV,P_XV)\nonumber\\
        &=& \frac{5}{\sigma^4} + G(P_XV,P_XV).
    \end{eqnarray*}
    Since $R^* = 0$, we obtain 
    \begin{equation}\label{5thm_eq_1}
        l(\sigma)\left(-\frac{2}{\sigma^3}\right) = \frac{5}{\sigma^4} + G(P_XV,P_XV).
    \end{equation}
    Next, let us calculate $G(R(V,X)X,V)$. We have
    \begin{eqnarray*}
        G(D_VD_XX,V) &=& G\left(k(\sigma)(D_VX),V\right)\nonumber\\
        &=& k(\sigma)G\left(-\frac{1}{\sigma}V + P_XV,V\right)\nonumber\\
        &=& k(\sigma)\left(-\frac{1}{\sigma^3} + G(P_XV,V)\right)\nonumber\\
        &=& k(\sigma)\left(-\frac{1}{\sigma^3}  + \frac{1}{2n\sigma}\frac{2n}{\sigma^2}\right) = 0,
    \end{eqnarray*}
    and
    \begin{eqnarray*}
        G(D_XD_VX,V) &=& XG(D_VX,V) - G(D_VX,D^*_XV)\nonumber\\
        &=& \partial_\sigma G\left(-\frac{1}{\sigma}V + P_XV,V\right) - G\left(\frac{V}{\sigma},\frac{V}{\sigma}\right) + G(P_XV,P_XV)\nonumber\\
        &=& \partial_\sigma\left(-\frac{1}{\sigma}\frac{1}{\sigma^2} + G(X,P_VV)\right) - \frac{1}{\sigma^2}\frac{1}{\sigma^2} + G(P_XV,P_XV)\nonumber\\
        &=& \partial_\sigma\left(-\frac{1}{\sigma^3} + \frac{1}{\sigma^3}\right) - \frac{1}{\sigma^4} + G(P_XV,P_XV) = -\frac{1}{\sigma^4} + G(P_XV,P_XV).
    \end{eqnarray*}
    Since $R = 0$, we have
    \begin{equation*}
       0 = -\frac{1}{\sigma^4} + G(P_XV,P_XV).
    \end{equation*}
    Combining this with (\ref{5thm_eq_1}) and (\ref{star5}), we obtain 
    \begin{equation*}
        l(\sigma)= -\frac{3}{\sigma},\quad k(\sigma) = \frac{1}{\sigma}.
    \end{equation*}
    The other case is shown in the same way.
\end{proof}
Note that these connections coincide with the $\alpha $-connections at $\alpha = \pm{1}$ in the Takano Gaussian space (recall Subsection 2.3).

\section{Discussion: Wasserstein Gaussian space}
By \emph{Wasserstein Gaussian space}, we mean the set of multivariate Gaussian distributions on $\mathbb{R}^n$ with mean zero equipped with the $L^2$-Wasserstein metric.
When we started investigating warped products in information geometry, we thought that we would be able to find dually flat connections on the Wasserstein Gaussian space and calculate its canonical divergence. 
The scenario we thought was the following. In \cite{takatsu}, it is proved that the Wasserstein Gaussian space has a cone structure. Recently in \cite{fujiwara2}, it is proved that we can find dually flat connections on
the space of density matrices equipped with the monotone metric. 
We can apply this result because the SLD metric and the Wasserstein metric on Gaussian distributions are essentialy the same on $\mathcal{D}$ \cite{bures}.
We thought that once we study dually flat affine connections on warped products, we would be able to extend the dually flat connections on the fiber space to the warped product in a natural way.
However, it turned out that it is difficult to draw dual affine coordinates of the dually flat connections on warped products we made. Thus, we do not know how to calculate the canonical divergence.
Let us explain the difficulty in this section.

In the previous sections, we discussed necessary conditions for warped products and fiber spaces to be dually flat. 
First, we show that it is also a sufficient condition for the Wasserstein Gaussian space. The question is, when we extend connections on the fiber space to the warped product, whether the warped product with those connections becomes dually flat or not.

According to the arguments in Section 4, we now define the connection $D$ on a cone $M = \mathbb{R}_{>0}\times_{f} F$, where $f(t) = t$ and $(F,g_F)$ is a Riemannian manifold. We consider the situation that the fiber space is equipped with a dually flat affine connection $\widetilde{\nabla}$.
    Let $G := g_B + f^2g_F$ be the warped product metric on $M$.
    Denote the base space by $(\mathbb{R}_{>0},g_B)$ with a coordinate $\{t\in\mathbb{R}_{>0}\}$ such that $g_B(\frac{\partial}{\partial t},\frac{\partial}{\partial t}) = 1$.
    Let $X:= \frac{\partial}{\partial t}$ and $\{U_i\}_{i=1}^n\subset\mathcal{L}(F)$ be a basis of $\mathcal{L}(F)$ such that $[U_i,U_j] = 0$ for any $i,j\in\{1,\ldots,n\}$. Following Theorem \ref{mainthm} and Lemma \ref{lem4_3}, define the connection $D$ by
    \begin{equation*}
        D_XX = \frac{1}{t}\frac{\partial}{\partial t},\quad  D_XV = D_VX := \frac{2}{t}V,\quad
        \begin{cases}
            \mbox{Hor } D_VW := 0,\\
            \mbox{Ver } D_VW := \mbox{Lift} (\widetilde{\nabla}_VW),
        \end{cases}
    \end{equation*}
    where $V,W$ are arbitrary vectors in $\{U_i\}_{i=1}^n$.

\begin{proposition}
    $(M,D,D^*)$ is a dually flat space.
\end{proposition}
\begin{proof}
    We only have to check that the curvature vanishes with respect to $D$. Let $U,V,W,Q$ be arbitrary vectors in $\{U_i\}_{i=1}^n$.
    Note that we have
    \begin{eqnarray}\label{last1}
        \mbox{Hor}D_UV  = 0,
    \end{eqnarray}
    and
    \begin{equation}\label{last2}
        D^*_XU = D^*_UX  = \frac{1}{t}U - \frac{1}{t}U = 0.
    \end{equation}
    We first check that $R(U,V)W$ vanishes.
    From (\ref{last1}) and (\ref{last2}), we have
    \begin{eqnarray*}
        G(R(U,V)W,X) &=& G(D_UD_VW,X) - G(D_VD_UW,X)\nonumber\\
        &=& \left\{UG(D_VW,X) - G(D_VW,D^*_UX)\right\} - \left\{VG(D_UW,X) - G(D_UW,D^*_VX)\right\}\nonumber\\
        &=& UG(\mbox{Hor}D_VW,X) - VG(\mbox{Hor}D_UW,X) = 0.
    \end{eqnarray*}
    Recall from Subsection 4.1 and (\ref{beta}) that
    \begin{eqnarray*}
        G(I\hspace{-.1em}I(V,W),I\hspace{-.1em}I^*(U,Q)) &=& \frac{\|\mbox{grad}f\|^2}{f^2}G(V,W)G(U,Q) - G(\mbox{Hor}(P_VW),\mbox{Hor}(P_UQ))  \nonumber\\
        &&\qquad+ \frac{G(V,W)G(\mbox{grad}f,P_UQ)}{f} - \frac{G(U,Q)G(\mbox{grad}f,P_VW)}{f}\nonumber\\
        &=& \frac{1}{t^2}G(V,W)G(U,Q) - G\left(\frac{G(V,W)}{t}\frac{\partial}{\partial t},\frac{G(U,Q)}{t}\frac{\partial}{\partial t}\right)  \nonumber\\
        &&\qquad+ \frac{1}{t}G(V,W)G\left(\frac{\partial}{\partial t},\frac{G(U,Q)}{t}\frac{\partial}{\partial t}\right) - \frac{1}{t}G(U,Q)G\left(\frac{\partial}{\partial t},\frac{G(V,W)}{t}\frac{\partial}{\partial t}\right)\nonumber\\
        &=& 0,
    \end{eqnarray*}
    thus we have $G(R(U,V)W,Q) = G(^FR(U,V)W,Q) = 0$ by (\ref{gauss1}). Hence, we have $R(U,V)W = 0$.
    
    Next, we check that $R(X,U)V$ vanishes. From (\ref{last1}) and (\ref{last2}), we have
    \begin{eqnarray*}
        G(R(X,U)V,X) &=& G(D_XD_UV- D_UD_XV,X)\nonumber\\
        &=& \{XG(D_UV,X) - G(D_UV,D^*_XX)\} - UG(D_XV,X) \nonumber\\
        &=& -UG\left(\frac{2}{t}V,X\right) = 0.
    \end{eqnarray*}
    We also have
    \begin{eqnarray*}
        G(R(X,U)V,Q) &=& \{XG(D_UV,Q) - G(D_UV,D^*_XQ)\} - \{UG(D_XV,Q) - G(D_XV,D^*_UQ)\}\nonumber\\
        &=& 2t\{g_F(\widetilde{\nabla}_UV,Q) - Ug_F(V,Q) + g_F(V,\widetilde{\nabla}^*_UQ)\}\nonumber\\
        &=& 0.
    \end{eqnarray*}
    Hence, $R(X,U)V = 0$. In a similar way, we can check $R(U,X)X = 0$.
\end{proof}

\begin{remark}
    We denote the Wasserstein Gaussian space over $\mathbb{R}^n$ by the \emph{$n\times n$ Wasserstein Gaussian space} since its elements are represented by $n\times n$ covariance matrices.
    For the $2\times 2$ Wasserstein Gaussian space, we remark that the existence of a dually flat affine connection $\widetilde{\nabla}$ is guaranteed by \cite{fujiwara2}.
\end{remark}

Thus the above proposition implies that we can furnish the $2\times 2$ Wasserstein Gaussian space with dually flat affine connections.
Though we think it necessary to draw dual affine coordinates to calculate the canonical divergence, it turned out to be difficult.
This is because, for example, $D_XU$ does not vanish, which means that the trivial extension of affine coordinates on the fiber space does not give affine coordinates of the warped product.
Here, trivial extension means $\{t,\xi_1,\ldots,\xi_n\}$ for affine coordinates $\{\xi_1,\cdots,\xi_n\}$ of the fiber space and the coordinate $\{t\}$ of the line.

\begin{remark}\label{finsler1}
In \cite{tayebi}, it is claimed that there is no dually flat proper doubly warped Finsler manifolds.
Let us restrict their argument to Riemannian manifolds.
For two manifolds $M_1$ and $M_2$ and their doubly warped product $(M_1\times M_2,G)$,
let $(x_i)$ and $(u_\alpha)$ be coordinates of $M_1$ and $M_2$, respectively.
Then their claim asserts that the coordinates $((x_i),(u_\alpha))$ on $(M_1\times M_2,G)$ cannot be affine coordinates for any dually flat connections on $M_1\times M_2$ unless $G$ is the product metric.
\end{remark}
For further understanding the relation between affine coordinates and their connections, let us observe the $2\times 2$ BKM cone (recall Subsection 4.4). 
Let $\bar{\nabla}$ be an affine connection whose affine coordinate is $\{a,b,c,d\}$, with which $2\times 2$ matrices are expressed as
\begin{equation*}
    \begin{pmatrix}
        a & c + ib\\
        c -ib & d
    \end{pmatrix}.
\end{equation*}
This $\bar{\nabla}$ is a dually flat affine connection on the BKM cone.
Let $\bar{D}$ be an affine connection whose affine coordinate is $\{t,\alpha,\beta,\gamma\}$, with which $2\times 2$ matrices are expressed as
\begin{equation*}
    t\begin{pmatrix}
        \alpha & \beta + i\gamma\\
        \beta - i\gamma& 1-\alpha
    \end{pmatrix}.
\end{equation*}
Relations of these coordinates are 
\begin{equation*}
    t = a + d,\quad\alpha = \frac{a}{a + d}, \quad1-\alpha = \frac{d}{a + d},
\end{equation*}
\begin{equation*}
    \frac{\partial}{\partial\alpha} = \begin{pmatrix}
        t & 0\\
        0 & -t
    \end{pmatrix} = (a+d)\left(\frac{\partial}{\partial a} - \frac{\partial}{\partial d}\right),
\end{equation*}
\begin{equation*}
    \frac{\partial}{\partial t} = \begin{pmatrix}
        \alpha & 0\\
        0 & 1-\alpha
    \end{pmatrix} = \frac{a}{a + d} \frac{\partial}{\partial a} + \frac{d}{a + d}\frac{\partial}{\partial d}.
\end{equation*}
Using these relations, we calculate
\begin{eqnarray*}
    \bar{\nabla}_{\frac{\partial}{\partial\alpha}}\frac{\partial}{\partial t} &=& \bar{\nabla}_{(a + d)\left(\frac{\partial}{\partial a}-\frac{\partial}{\partial d}\right)}\left(\frac{a}{a + d} \frac{\partial}{\partial a} + \frac{d}{a + d}\frac{\partial}{\partial d}\right)\nonumber\\
    &=& (a+d)\left(\frac{\partial}{\partial a}-\frac{\partial}{\partial d}\right)\left(\frac{a}{a + d}\right)\frac{\partial}{\partial a} + (a+d)\left(\frac{\partial}{\partial a}-\frac{\partial}{\partial d}\right)\left(\frac{d}{a + d}\right)\frac{\partial}{\partial d}\nonumber\\
    &=& (a+d)\left(\frac{d}{(a + d)^2}+\frac{a}{(a+d)^2}\right)\frac{\partial}{\partial a} + (a+d)\left(-\frac{d}{(a+d)^2}-\frac{a}{(a+d)^2}\right)\frac{\partial}{\partial d}\nonumber\\
    &=& \frac{\partial}{\partial a} -\frac{\partial}{\partial d}\neq 0.
\end{eqnarray*}
On the other hand
\begin{equation*}
    \bar{D}_{\frac{\partial}{\partial\alpha}}\frac{\partial}{\partial t} = 0.
\end{equation*}
Hence, $\bar{\nabla}$ and $\bar{D}$ are different.
\\
\section{Appendix}
Main contributions of this appendix are following two points.
\begin{itemize}
    \item We study an example of warped product whose dually flat connections are not realized as ($\pm{1}$)-connections of $\alpha$-connections. 
    \item We study dually flat connections compatible with the structure of a two-dimensional warped product.
\end{itemize}
\subsection{Preliminaries for elliptic distributions}
As described in \cite{elliptic}, a $p$-dimensional random variable $X$ is said to have an elliptic distribution with parameters $\mu^{\mathrm{T}} = (\mu_1,\cdots, \mu_p)$ and $\Psi$, a $p\times p$ positive definite matrix, if its density is 
\begin{equation*}
    p_h(x|\mu,\Psi) = \frac{h\{(x-\mu)^{\mathrm{T}}\Psi^{-1}(x-\mu)\}}{\sqrt{\det\Psi}}
\end{equation*}
for some function $h$. We say that $X$ has an $EL_p^h(\mu,\Psi)$ distribution.

We consider the class of one-dimensional elliptic distributions $EL_1^h(\mu,\sigma^2)$, where $\theta = (\mu,\sigma)$. We set $Z$ as an $EL_1^h(0,1)$ random variable and $ W = \{d \log h(Z^2)\}/d(Z^2)$.
We also set 
\begin{equation*}
    a = E(Z^2W^2),\quad b = E(Z^4W^2), \quad d = E(Z^6W^3).
\end{equation*}

The Fisher metric of elliptic distributions is
\begin{eqnarray}\label{fisher}
    ds^2 =  \frac{4a d\mu^2 + (4b-1)d\sigma^2}{\sigma^2}.
\end{eqnarray}
We denote the Fisher metric $ds^2$ as $G_F$.

\begin{example}
    Gaussian distribution, Student's t distribution and Cauchy distributions are examples of elliptic distributions. Their constants are given in Table \ref{table:data_type}, which is calculated in \cite{elliptic}.
    \begin{table}[hbtp]
        \caption{Important constants}
        \label{table:data_type}
        \centering
        \begin{tabular}{lccc}
          \hline
            & Gauss  & Cauchy   &  Student's t  \\
          \hline \hline
          $a$  & $\frac{1}{4}$ & $\frac{1}{8}$ & $\frac{k + 1}{4(k + 3)}$ \\
          $b$  & $\frac{3}{4}$   & $\frac{3}{8}$ & $\frac{3(k + 1)}{4(k + 3)}$\\
          $d$  & $\frac{-15}{8}$  & $-\frac{5}{16}$  & $-\frac{15(k + 1)^2}{8(k + 3)(k + 5)}$\\
          \hline
        \end{tabular}
      \end{table}
\end{example}
\subsection{Elliptic distributions as warped products}
We set
\begin{equation*}
    t := \sqrt{4b-1}\log\sigma.
\end{equation*}
Since
\begin{equation*}
    G_F\left(\frac{\partial}{\partial t},\frac{\partial}{\partial t}\right) = G_F\left(\frac{\partial\sigma}{\partial t}\frac{\partial}{\partial \sigma},\frac{\partial\sigma}{\partial t}\frac{\partial}{\partial \sigma}\right) = \frac{\sigma^2}{4b-1}G_F\left(\frac{\partial}{\partial\sigma},\frac{\partial}{\partial\sigma}\right) = 1, 
\end{equation*}
we have
\begin{equation*}
    G_F = dt^2 + f(t)^2d\mu^2,
\end{equation*}
where
\begin{equation*}
    f(t) := \sqrt{4a}\exp\left(-\frac{t}{\sqrt{4b-1}}\right).
\end{equation*}

\subsection{Calculations of $R(V,X,X,V)$}
For the parameter spaces $M$ of elliptic distributions, we set new assumptions.
\begin{assume}\label{assume_appendix}
    Let $B = \mathbb{R}_{>0}$ with the Euclidean metric $g_B$ such that $g\left(\frac{\partial}{\partial t},\frac{\partial}{\partial t}\right) = 1$, $f(t) = \sqrt{4a}\exp\left(-\frac{t}{\sqrt{4b-1}}\right)$ and $F = \mathbb{R}$ with the Euclidean metric
     and $\widetilde{\nabla},\widetilde{\nabla}^*$ be a dually flat affine connections on $\mathbb{R}$ with the Euclidean metric $g_\mu$ such that $g_\mu\left(\frac{\partial}{\partial\mu},\frac{\partial}{\partial\mu}\right) = 1$. For an arbitrary connection $D$ on $B\times_f F$, we assume that $D$ satisfies
     \begin{itemize}
        \item $D_XY \mbox{ is horizontal,  } i.e. ,\, D_XY\in\mathcal{L}(B)$ for any $X,Y\in\mathcal{L}(B)$,
        \item $\mbox{Ver }(D_VW) = \mbox{Lift }(\widetilde{\nabla}^F_VW)$ for any $V,W\in\mathcal{L}(F)$.
    \end{itemize}
    We also assume $R = 0$ where $R$ is the curvature of $M$ with respect to $D$.
\end{assume}
We next calculate the curvature $R$ under Assumption \ref{assume_appendix}. We set the notations $X = \frac{\partial}{\partial t}, V = \frac{\partial}{\partial \mu}$ and $k,l$ as 
\begin{equation*}
    D_{\frac{\partial}{\partial t}}\frac{\partial}{\partial t} = k(t)\frac{\partial}{\partial t}
\end{equation*}
and 
\begin{equation*}
    P_XV = l(t,\mu)\frac{\partial}{\partial \mu},
\end{equation*}
where $P_XV := \frac{1}{2}(D_XV - D^*_XV)$.
Since 
\begin{equation*}
    \frac{Xf}{f} = \frac{\sqrt{4a}\left(-\frac{1}{\sqrt{4b-1}}\right)\exp\left(-\frac{t}{\sqrt{4b-1}}\right)}{\sqrt{4a}\exp\left(-\frac{t}{\sqrt{4b-1}}\right)} = -\frac{1}{\sqrt{4b-1}},
\end{equation*}
we have
\begin{eqnarray*}
    G_F( D_VD_XX,V) &=& k(t)\left\{-\frac{1}{\sqrt{4b-1}}f(t)^2 + l(t,\mu)f(t)^2\right\},
\end{eqnarray*}
\begin{eqnarray*}
    G_F( D^*_VD^*_XX,V) &=& -k(t)\left\{ -\frac{1}{\sqrt{4b-1}}f(t)^2 - l(t,\mu)f(t)^2\right\},
\end{eqnarray*}
\begin{eqnarray*}
    G_F( D_XD_VX,V) &=& XG_F( \frac{Xf}{f}V + P_XV,V) - G_F( D_XV,D^*_XV)\\
    &=& \frac{\partial}{\partial t}\left\{\left(-\frac{1}{\sqrt{4b-1}}f(t)^2 + l(t,\mu)f(t)^2\right)\right\} - \left\{\left(\frac{Xf}{f}\right)^2G_F( V,V)- G_F( P_XV,P_XV)\right\}\\
    &=& f(t)^2\left\{\frac{1}{4b-1} + \partial_t l  - \frac{2l}{\sqrt{4b-1}} + l^2\right\}
\end{eqnarray*}
and
\begin{eqnarray*}
    G_F (D^*_XD^*_VX, V ) &=& f^2\left\{\frac{1}{4b-1} - \partial_tl + \frac{2l}{\sqrt{4b-1}} + l^2\right\}.
\end{eqnarray*}
Hence, we have
\begin{eqnarray*}
    R(V,X,X,V) &=& G_F( D_VD_XX,V) - G_F( D_XD_VX , V)\\
    &=& f(t)^2\left\{-\frac{k(t)}{\sqrt{4b-1}} + kl - \frac{1}{4b-1} - \partial_t l + \frac{2l}{\sqrt{4b-1}} - l^2\right\}
\end{eqnarray*}
and
\begin{eqnarray*}
    R^*(V,X,X,V) &=& G_F( D^*_VD^*_XX,V)- G_F( D^*_XD^*_VX, V )\\
    &=& f(t)^2\left\{ \frac{k(t)}{\sqrt{4b-1}}  + kl - \frac{1}{4b-1} + \partial_t l - \frac{2l}{\sqrt{4b-1}} - l^2\right\}.
\end{eqnarray*}
Since we now consider the case $R = R^* = 0$, we have
\begin{eqnarray*}
        &&f(t)^2\left\{kl - \frac{1}{4b-1} - l^2\right\} = 0,\\
        &&f(t)^2\left\{-\frac{k(t)}{\sqrt{4b-1}} - \partial_t l + \frac{2l}{\sqrt{4b-1}}\right\} = 0.
\end{eqnarray*}
 Summarizing the above arguments, we have the following theorem.
\begin{theorem}\label{constant_connection}
    For any $\gamma\in\mathbb{R}$, we define affine connections $D, D^*$ on $M$ by
\begin{eqnarray*}
    &&D_{\frac{\partial}{\partial t}}\frac{\partial}{\partial t} = k\frac{\partial}{\partial t},\\
    &&D_{\frac{\partial}{\partial \mu}}\frac{\partial}{\partial\mu} = \nabla_{\frac{\partial}{\partial t}}\frac{\partial}{\partial\mu} + l\frac{\partial}{\partial\mu},\\
    &&D_{\frac{\partial}{\partial\mu}}\frac{\partial}{\partial\mu} = \nabla_{\frac{\partial}{\partial\mu}}\frac{\partial}{\partial\mu} + \gamma\frac{\partial}{\partial\mu} + \frac{l}{2\sigma^2}\frac{\partial}{\partial t},
\end{eqnarray*}
where $\nabla$ is the Levi-Civita connection and $k,l$ are arbitrary functions satisfying
\begin{eqnarray}\label{difeq}
    \begin{cases}
        kl - \frac{1}{4b-1} - l^2 = 0,\\
        -\frac{k(t)}{\sqrt{4b-1}} - \partial_t l + \frac{2l}{\sqrt{4b-1}} = 0.
    \end{cases}
\end{eqnarray}
Then $(M,D,D^*)$ satisfies Assumption \ref{assume_appendix}.
\end{theorem}
\begin{remark}
    In Section 5, we consider Takano Gaussian space $(L^{n + 1},G_T,\nabla^{(\alpha)})$. 
    It was shown in Lemma \ref{prop5_4} and Theorem \ref{takano_2} that the dually flat connections $D,D^*$ on $(L^{n + 1},G_T)$ compatible with warped product structure satisfy following two equations:
    \begin{equation}\label{takano_connection_1}
        \mbox{Hor} P_{\frac{\partial}{\partial m_i}}\frac{\partial}{\partial m_i} = \frac{1}{2n\sigma}\frac{\partial}{\partial\sigma} \quad \mbox{or} \quad -\frac{1}{2n\sigma}\frac{\partial}{\partial\sigma},
    \end{equation}
    \begin{equation}\label{takano_connection_2}
        D_{\frac{\partial}{\partial\sigma}}\frac{\partial}{\partial\sigma} = \frac{1}{\sigma}\frac{\partial}{\partial \sigma} \quad  \mbox{or}\quad -\frac{3}{\sigma}\frac{\partial}{\partial\sigma}.
    \end{equation}
    Since the Fisher metric $G_T$ of Takano Gaussian space is expressed as
    \begin{equation*}
        G_T = \frac{dm_1^2 + \cdots + dm_n^2 + 2nd\sigma^2}{\sigma^2},
    \end{equation*}
    if we set the parameter $t = \sqrt{2n}\log \sigma $, we have
    \begin{equation*}
        G_T\left(\frac{\partial}{\partial t},\frac{\partial}{\partial t}\right) = G_T\left(\frac{\partial\sigma}{\partial t}\frac{\partial}{\partial\sigma},\frac{\partial\sigma}{\partial t}\frac{\partial}{\partial \sigma}\right) = \frac{\sigma^2}{2n}\frac{2n}{\sigma^2} = 1.
    \end{equation*}
    Hence, the metric is
    \begin{equation*}
        G_T = dt^2 + f_T(t)^2 \{dm_1^2 + \cdots + dm_n^2\},
    \end{equation*}
    where
    \begin{equation*}
        f_T(t) = \exp\left(-\frac{t}{\sqrt{2n}}\right).
    \end{equation*}
    From (\ref{takano_connection_1}) and (\ref{takano_connection_2}), we have
    \begin{equation*}
        k = \sqrt{\frac{2}{n}},\quad l = \frac{1}{\sqrt{2n}}.
    \end{equation*}
    This is the only solution of (\ref{difeq}) if $l$ is constant when $n = 1, a = \frac{1}{4}$ and $ b = \frac{3}{4}$.
\end{remark}
\subsection{$\alpha$-connections and dually flat connections}
Using calculations for $\alpha$-connections on elliptic distributions $\nabla^{(\alpha)}$ in \cite{elliptic}, we have 
\begin{eqnarray*}
    \nabla^{(\alpha)}_{\frac{\partial}{\partial\sigma}}\frac{\partial}{\partial\sigma} &=& \frac{1-4b + \alpha(6b + 4d - 1)}{(4b-1)\sigma}\frac{\partial}{\partial\sigma}\\
    &=& \frac{1-4b + \alpha(6b + 4d - 1)}{(4b-1)\sigma}\frac{\partial t}{\partial\sigma}\frac{\partial}{\partial t}.
\end{eqnarray*}
On the other hand, since $\frac{\partial t}{\partial\sigma} = \frac{\sqrt{4b-1}}{\sigma}$, we have
\begin{eqnarray*}
    \nabla^{(\alpha)}_{\frac{\partial}{\partial\sigma}}\frac{\partial}{\partial\sigma} &=& \nabla^{(\alpha)}_{\frac{\partial t}{\partial\sigma}\frac{\partial}{\partial t}}\frac{\partial t}{\partial \sigma}\frac{\partial}{\partial t}\\
    &=& \frac{\partial t}{\partial \sigma}\sqrt{4b-1}\frac{(-1)}{\sigma^2}\frac{\partial\sigma}{\partial t}\frac{\partial }{\partial t} + \left(\frac{\partial t}{\partial \sigma}\right)^2 \nabla^{(\alpha)}_{\frac{\partial}{\partial t}}\frac{\partial}{\partial t} .
\end{eqnarray*}
Comparing two equations above, we have
\begin{equation*}
    \nabla^{(\alpha)}_{\frac{\partial}{\partial t}}\frac{\partial}{\partial t} = \frac{\alpha(6b + 4d - 1)}{(4b-1)^{\frac{3}{2}}}\frac{\partial}{\partial t}.
\end{equation*}
The only solution of $(\ref{difeq})$ when $l$ is constant is 
\begin{equation}\label{only_constant_sol}
    k = \frac{2}{\sqrt{4b-1}},\quad l= \frac{1}{\sqrt{4b-1}}.
\end{equation}
If we set $k,l$ as (\ref{only_constant_sol}), dually flat connections which are compatible with the structure of warped products are constructed.

For dually flat connections $D,D^*$ constructed by Theorem \ref{constant_connection} using (\ref{only_constant_sol}), we have 
\begin{equation*}
    \frac{1}{2}(D_{\partial_t}\partial_t - D^*_{\partial_t}\partial_t) = \frac{2}{\sqrt{4b-1}}\frac{\partial}{\partial t}.
\end{equation*}
On the other hand, for $\alpha$-connections of elliptic distributions $\nabla^{(\alpha)}$, we have
\begin{equation*}
    \frac{1}{2}(\nabla^{(-\alpha)}_{\partial_t}\partial_t - \nabla^{(\alpha)}_{\partial_t}\partial_t) = \frac{-\alpha(6b + 4d - 1)}{(4b-1)^{\frac{3}{2}}}\frac{\partial}{\partial t}.
\end{equation*}
\begin{example}
    We compare dually flat connections on the line $B = \mathbb{R}_{>0}$ constructed by Theorem \ref{constant_connection} with dually flat connections among $\alpha$-connections in Table \ref{table:compare}.
    Note that $\alpha$-connections of Cauchy distribution are not dually flat connections \cite{elliptic} and 
Student's $t$ distributions are dually flat when $\alpha = \pm{\frac{k + 5}{k-1}}$.
\begin{table}[hbtp]
    \caption{Comparing dually flat connections}
    \label{table:compare}
    \centering
    \begin{tabular}{lllll}
      \hline
        Connections on $B = \mathbb{R}_{>0}$ & Gauss  & Cauchy   &  Student's t  \\
      \hline \hline
      $\frac{2}{\sqrt{4b-1}}\partial_t$  & $\sqrt{2}\partial_t$ & $2\sqrt{2}\partial_t$ & $\sqrt{\frac{2(k+3)}{k}}\partial_t$\\
      $\frac{-\alpha(6b + 4d - 1)}{(4b-1)^{\frac{3}{2}}}\partial_t$  &$\sqrt{2}\partial_t (\alpha = 1)$   & none & $\frac{k\sqrt{2}}{k + 3}\partial_t (\alpha = \frac{k + 5}{k-1})$\\
      \hline
    \end{tabular}
  \end{table}

\end{example}

Although every dually flat connections which is compatible with structure of warped product appeared before this appendix were realized as one of $\alpha$-connections, it is observed in this appendix that it does not always happen. 

\begin{remark}
    In \cite{furuhata}(4.2), they defined $\alpha$-connections with respect to a metric $g = \frac{dx^2 + \lambda^2 dy^2}{y^2}$ ($\lambda > 0$) on the upper half plane $\{(x,y) | x\in\mathbb{R}, y > 0\}$.
    By direct calculations, we see that their $\alpha$-connections are also compatible with the structure of warped product and their $\alpha$-connections are dually flat when $\alpha = \pm{1}$.
\end{remark}

\begin{acknowledgements}
    The author wishes to thank his supervisor
     Shin-ichi Ohta for his support and encouragement. He would also like to express his gratitude to Akio Fujiwara for many valuable discussions and Hiroshi Matsuzoe for helpful comments.
     He also thanks Masaki Yoshioka for fruitful conversations on the topic in Section 7.
\end{acknowledgements}

\end{document}